\documentclass{article}
\usepackage{amsmath,amsfonts,amssymb,amsthm,graphicx}
 \usepackage{dsfont}     
\usepackage{enumerate} 
\usepackage{color}
\usepackage[round]{natbib}
\usepackage{multirow}
\usepackage{comment}

\usepackage{booktabs,subcaption,dcolumn}

\usepackage{algorithm}
\usepackage[noend]{algpseudocode}  
\usepackage[colorlinks=true,citecolor=blue,pdfpagemode=UseNone,pdfstartview=FitH]{hyperref}

\usepackage[margin=1.2in]{geometry}

\renewcommand{\d}{\,\mathrm{d}}

\newcommand{\p}{\mathbb{P}}

\newcommand{\E}{\mathbb{E}}    
\newcommand{\R}{\mathbb{R}}    

  \newcommand{\id}{\mathds{1}} 
 
\newcommand{\cD}{\mathcal{D}}
 
\newcommand{\cG}{\mathcal{G}}

\usepackage{setspace}
\newcommand{\com}[1]{\marginpar{{\begin{minipage}{0.26\textwidth}{\setstretch{1.1} \begin{flushleft} \footnotesize \color{red}{#1} \end{flushleft} }\end{minipage}}}}

\theoremstyle{plain}
\newtheorem{theorem}{Theorem}

\newtheorem{lemma}{Lemma}
\newtheorem{proposition}{Proposition}

\theoremstyle{definition}

\newtheorem{example}{Example}

\newtheorem{fact}{Fact}

\theoremstyle{remark}
\newtheorem{remark}{Remark}

 \renewcommand{\cite}{\citet}  

\title{False discovery rate control with e-values} 
\author{
  Ruodu Wang\thanks%
  {Department of Statistics and Actuarial Science,
  University of Waterloo.
  E-mail: \href{mailto:wang@uwaterloo.ca}{wang@uwaterloo.ca}.}
\and  
Aaditya Ramdas\thanks%
    {Departments of Statistics and Machine Learning,
    Carnegie Mellon University.
    E-mail: \href{mailto:aramdas@cmu.edu}{aramdas@cmu.edu}.}
}

\begin{document}

\maketitle

\begin{abstract}
E-values have gained attention as potential alternatives to p-values as measures of uncertainty, significance and evidence. In brief, e-values are realized by random variables with expectation at most one under the null; examples include betting scores, (point null) Bayes factors, likelihood ratios and stopped supermartingales. We design a natural analog of the Benjamini-Hochberg (BH) procedure for false discovery rate (FDR) control that utilizes e-values, called the e-BH procedure, and compare it with the standard procedure for p-values. One of our central results is that, unlike the usual BH procedure, the e-BH procedure controls the FDR at the desired level---with no correction---for any dependence structure between the e-values. We illustrate that the new procedure is convenient in various settings of complicated dependence, structured and post-selection hypotheses, and multi-armed bandit problems. Moreover, the BH procedure is a special case of the e-BH procedure through calibration between p-values and e-values. Overall, the e-BH procedure is a novel, powerful and general tool for multiple testing under dependence, that is complementary to the BH procedure, each being an appropriate choice in different applications. 

~

\end{abstract}


 \section{Introduction}
 \label{sec:1}

We study procedures for controlling the false discovery rate (FDR) as in \cite{BH95}. We will encounter p-values in this paper, but focus more on  e-values and e-tests as in \cite{VW20}. 
We use ``e-value" as an abstract umbrella term which encompasses betting scores, likelihood ratios, and stopped supermartingales, which appear in the recent literature, e.g., \cite{S20, GDK20,WRB20} and \cite{HRMS20,HRMS21}. (Since Bayes factors~\citep{KR95} for point nulls are also e-values, one may hope that e-values appeal to adherents of frequentist, Bayesian and game-theoretic foundations of probability, but this paper abstains from further philosophical discussion, treating e-values as a useful technical tool that arises naturally in many situations.)  

Throughout, let $H_1,\dots,H_K$ be $K$ hypotheses, and write $\mathcal K=\{1,\dots,K\}$.
Let the true (unknown) data-generating  probability measure be denoted by $\p$.
For each $k\in \mathcal K$, 
it is useful to think of hypothesis $H_k$ as implicitly defining a set of joint probability measures, and $H_k$ is called a true null hypothesis if $\p \in H_k$.

Following \cite{VW20}, 
a \emph{p-variable} $P$ is a  random variable  that satisfies $\p(P\le \alpha)\le \alpha$ (often with equality) for all $\alpha \in (0,1)$. 
In other words, a p-variable is stochastically larger than $\mathrm {U}[0,1]$ (values of $P$ larger than $1$ can be treated as $1$). An \emph{e-variable} $E$ is a $[0,\infty]$-valued random variable satisfying $\E[E]\le1$. 
E-variables are often obtained from stopping an \emph{e-process} $(E_t)_{t \geq 0}$, which is a nonnegative stochastic process adapted to a pre-specified filtration such that $\mathbb{E}[E_\tau] \leq 1$ for any stopping time $\tau$ (an example would be a supermartingale with initial value $1$).

Let $\mathcal N\subseteq \mathcal K$ be the set of indices of true null hypotheses, which is unknown to the decision maker, and $K_0$ be the number of true null hypotheses, thus the cardinality of $\mathcal N$. The ratio $K_0/K$ may be close to $1$, meaning that the signals are sparse.

 For reasons to use  e-values over p-values, see \cite{S20}, \cite{VW20} and \cite{GDK20}; however, we summarize our own perspectives later in Section \ref{sec:motivation}. 
We consider two  settings of testing multiple hypotheses, and sometimes convert between them:
\begin{enumerate}
\item  For each $k\in \mathcal K$, 
  $H_k$ is  associated with p-value $p_k$, which is a realization of a random variable $P_k$. If $k\in \mathcal N$, then $P_k$ is a p-variable.  
\item For each $k\in \mathcal K$, 
  $H_k$ is  associated with e-value $e_k$, which is a realization of a random variable $E_k$. If $k\in \mathcal N$, then $E_k$ is an e-variable.  
\end{enumerate}

A \emph{p-testing procedure} $\cD:[0,1]^K\to 2^{\mathcal K}$ (resp.~an \emph{e-testing procedure} $\cD:[0,\infty]^K\to 2^{\mathcal K}$) 
gives the indices of rejected hypotheses based on observed p-values (resp.~e-values).  
We tacitly require that all testing procedures are Borel functions. 
The terms ``p-values/e-values" refer to both the random variables and their realized values; these should be clear from the context. 
  
 Let $\cD$ be a  p-testing procedure or an e-testing procedure. 
The rejected hypotheses by  $\cD$ are called discoveries. 
We write $F_{\cD}:=|\cD \cap \mathcal N|$ as the number of true null hypotheses that are rejected (i.e., false discoveries), and $R_{\cD}:=|\cD|$ as the total number of discoveries. 
 The value of interest is   $F_{\cD}/R_{\cD}$, called the false discovery proportion (FDP), which is the ratio of the number of false discoveries to that of all claimed discoveries, with the convention $0/0=0$ (i.e., FDP is $0$ if there is no discovery). Since both quantities  $F_{\cD}$ and $R_{\cD}$ are random,  \cite{BH95} proposed to control FDR, which is the expected value of FDP, that is,
  $ 
 \mathrm{FDR}_{\cD}:=\E [ {F_{\cD}}/{R_{\cD} } ],
$ 
where the expected value is taken  under the true probability.   
Other ways of controlling the false discovery rather than the FDR are studied by, for instance, 
\cite{Genovese/Wasserman:2004,Genovese/Wasserman:2006JASA} and \cite{Goeman/Solari:2011} with p-values, and  \cite{Vovk19} with e-values. We focus on controlling FDR  in this paper for its popularity in modern sciences.


For those unfamiliar with e-values, we offer a couple of quick remarks. First, an e-variable $E$ can be converted to a p-variable $P = 1/E$; the validity of $P$ can be checked using Markov's inequality. However, since Markov's inequality is not tight, such a translation is not airtight. A p-variable $P$ can also be converted to an e-variable, but this is more complex; one simple example is to set $E = P^{-1/2} - 1$ (and more generally $f(P)$ is an e-variable if $\int_0^1 f(u)\d u \leq 1$). Importantly, small p-values correspond to large e-values, so we will reject e-values above some threshold.

\subsection*{A brief summary of our contributions}
 
For $k\in \mathcal K$, let $e_{[k]}$ be the $k$-th order statistic of $e_1,\dots,e_K$, from \emph{the largest to the smallest}. 
We design a simple e-value analog of the standard BH procedure, which will be called the \emph{base e-BH procedure}.  
For $\alpha>0$, define the e-testing procedure $\cG(\alpha):[0,\infty]^K\to 2^{\mathcal K}$ which rejects hypotheses with the largest $k_e^*$ e-values,
where 
\begin{equation} 
\label{eq:e-k-intro} 
k_e^*=\max\left\{k\in \mathcal K: \frac{k e_{[k]}}{K} \ge \frac{1}{\alpha}\right\},\end{equation} 
with the convention $\max(\varnothing)=0$, 
and accepts the rest.

The astute reader may note that since $1/e_k$ is a valid p-value, the base e-BH procedure is simply the BH procedure applied to the corresponding p-values. Hence, at first glance, apparently nothing is gained with the e-value viewpoint. 
However, one of our central results is the surprising property that \begin{quote} the base e-BH procedure controls FDR at level $\alpha$ even under unknown \emph{arbitrary dependence} between the e-values. \end{quote} It is well known that such a statement is not true for general arbitrarily dependent p-values. 
Moreover, the \emph{full e-BH procedure}, to be specified in Section~\ref{sec:4}, involves a pre-screening of e-values, allowing us to ``boost'' them up by a factor before feeding them to the base e-BH procedure.

We also derive several important, but secondary, results. These include an analysis of self-consistent e-testing procedures, the fact that known results on FDR control for BH and BY \citep{BY01} procedures can be derived as special cases of the e-BH procedure, and the optimality of e-BH procedure amongst all FDR controlling procedures acting on arbitrarily dependent e-values.

Finally, the paper contains several other tertiary contributions, including a mathematical study of how conservative the e-BH procedure is under arbitrary dependence, independence and PRDS (positive regression dependence on a subset). We provide a few interesting motivating examples along the way: identifying skillful traders (finance) and
identifying promising slot machines by adaptive sampling (multi-armed bandits). Simulation results on the e-BH procedure are produced in Section \ref{sec:simulation} and Appendix~\ref{app:num}. 
A real-data analysis on  returns of cryptocurrencies is produced in Section \ref{sec:realdata}.
Proofs of most results are put in Appendix \ref{app:pf}. 

Before delving into the above, it seems critical to first answer a motivating question: why use e-values in the first place?

\section{When might one prefer e-values over p-values?}\label{sec:motivation}

Given that the Benjamini-Hochberg procedure using p-values is so widely established and accepted in the sciences, one must first answer the question of why one should care about FDR control (or multiple testing) with e-values. This is a philosophical, subjective, question and the opinions of the authors (and readers) are complementary to the contributions of this paper. 
For example, in his upcoming discussion paper, \cite{S20} eloquently provides various arguments for the use of ``betting scores", which are e-values with a game-theoretic interpretation, as measures of evidence and uncertainty in the scientific enterprise while \cite{GDK20} find them to be a middle ground between Bayesian and frequentist (and financial, when combined with the earlier view) interpretations of evidence.
Other authors have also already weighed in on the subject, but rather than deferring to them, we wish to offer our own perspectives in light of recent research. 

Below, we list some situations when one may potentially prefer to use e-values over p-values for non-philosophical, purely statistical, reasons. To avoid confusion, we note that \emph{the following discussion relates to the construction of a single e-value or a single p-value}---in other words, it applies equally to single and multiple hypothesis testing.

\begin{enumerate}
    \item \textbf{High-dimensional asymptotics}. One of the most classical ways to compute p-values is to use the asymptotic distribution of the likelihood ratio test statistic, as given by Wilks' theorem (\cite{Wilks38}). However, the correctness of Wilks' theorem is typically justified when the dimensionality $d$ of the data remains fixed, and the sample size $n$ tends to infinity. There are several results on the high-dimensional asymptotics of likelihood ratios, but the resulting p-value relies on the practitioner making assumptions on the relative scalings of $n$ and $d$; see \cite{JY13} for example. In contrast,  the likelihood ratio for a point null hypothesis is a valid e-value in finite samples, meaning that its expectation equals one under the null regardless of $d$ or $n$. The same holds for mixtures (over the alternative) of likelihood ratios as well.
    \item \textbf{Irregular (composite) models}. There are many composite null hypothesis testing problems for which we may know of no direct way to construct a valid p-value even under low-dimensional asymptotics---this could happen because the model is singular or irregular and Wilks' theorem fails to hold \citep{D09} and in such cases the validity of the bootstrap is also typically unknown. Recently, the \emph{split} likelihood ratio statistic~\citep{WRB20} was developed, along with several other variants, and shown to yield an e-value under no assumptions on the composite null, or on $d$ and $n$. Examples of new settings in which one can now construct e-values under no regularity assumptions include mixtures (e.g., testing if data comes from a mixture of $\leq k$ versus $>k$ components), shape-constraints (e.g., log-concavity), dependence structures (e.g., multivariate total positivity), and several latent variable models.
    \item \textbf{Sequential inference}. The correctness of Wald's sequential likelihood ratio test is based on noting that the likelihood ratio process $(L_t)$ is a nonnegative martingale with initial value one. Thus, at any stopping time $\tau$, the optional stopping theorem implies that $L_\tau$ is an e-value since its expectation is at most one; in other words, $(L_t)$ is an e-process. Moving beyond parametric settings, sequential inference is often enabled by designing nonparametric supermartingales~\citep{HRMS20} that immediately yield e-values at the stopping time. These have been called test martingales in prior work~\citep{SSVV11}. In summary, e-values arise very naturally in sequential inference as stopped {e-processes}. 
    For point nulls, e-processes can always be dominated by nonnegative martingales (meaning they can be increased to a martingale without threat to validity). But for composite nulls, there is a big and important difference between e-processes, supermartingales and martingales. See~\cite{ramdas2020admissible,ramdas2021can} for details.

    \item \textbf{Accumulation of information and evidence}. 
    Suppose the ultimate goal of a scientist is to either reject (with level $\alpha$) or not reject a given hypothesis \emph{in a single run}. Then, the most powerful method is to reject a p-value $p$ no larger than $\alpha$. This testing procedure can also be  achieved by a simple e-value  $e:= \id_{\{p\le \alpha\}} / \alpha$, using a threshold $1/\alpha$ (called an all-or-nothing bet by~\cite{S20}); in fact, no other e-values would be more powerful in general. The situation becomes quite different when the hypothesis (if seen as promising by the first experiment) will be tested with future evidence and possibly by other scientists: the all-or-nothing e-value carries little information for the next studies, whereas a generic e-value can provide a continuum of evidence strength. Even if the first e-value  is moderate (e.g., $e_1=3$), it is useful and can be easily merged into the next experiment~\citep{VW20}, for example by simply multiplying the two e-values (this is valid if $\mathbb{E}[E_2|E_1=e_1]\leq 1$ under the null) or averaging them (always valid). On the other hand, a p-value of $0.1$ is difficult to use for future studies that may depend on data from the current one, and may be discarded (or worse, ``p-hacked") due to the all-or-nothing nature of the p-test or the difficulty of merging p-values under such sequential dependence as above. Hence, for a \emph{dynamic flow} of experiments, common in modern sciences, it is beneficial to track and report e-values --- see \cite{S20} and \cite{GDK20}, as well as their counterpart in estimation,  confidence sequences~\citep{HRMS21}.
    
    \item \textbf{Robustness to misspecfication}. In genetics, it is not uncommon to encounter p-values with astronomically small values (like $10^{-20}$), or sometimes point masses near the value one, even though the sample sizes may not intuitively support such extreme evidence. This is often (but not always) reflective of utilizing a model that is not perfectly specified. The validity of p-values is quite sensitive to model misspecification, because they utilize the entire (hypothesized) distribution of the test statistic. In contrast, e-values can be constructed without over-reliance on fine-grained tail information, and so they are typically more robust than p-values to misspecification (but less powerful under perfect specification). For example, instead of assuming that the data $X$ is Gaussian to build a p-value, we may instead assume that it is symmetric about the origin (under no further moment constraints), in which case $\exp(\lambda X - \lambda^2 X^2/2)$ is a valid e-value for any $\lambda \in \mathbb R$. The proof, by Bahadur and Eaton, was communicated by \cite{E69}. (However, it is inadmissible; admissible improvements can be found in~\cite{ramdas2020admissible}.)
    
    \item \textbf{Robustness to dependence}. To construct a single valid p-value, it is often assumed (for convenience of deriving limiting distributions) that the underlying observations are independent, and indeed their validity is often hurt if this assumption is violated. However, we can often construct an e-value quite easily in settings where the observations are dependent, and this is because we are requiring less of an e-value (just bounded in expectation) than of a p-value (knowledge of its whole distribution). As one example, suppose we observe non-negative data $X_1,\dots,X_n$, and we wish to test $H_0: \mathbb E[X_i] \leq \mu$ for all $i$. Then, $(X_1 +\dots+X_n)/(n\mu)$ is an e-value for any dependence structure of   $X_1,\dots,X_n$, and does not require making any distributional assumption. 
    In addition, combination of multiple p-values relies heavily on dependence assumptions \citep{VW19}, unlike that of e-values \citep{VW20}.
    In the field of risk management,  analysis of risks with unknown or complicated dependence has recently been an active topic; see e.g., \citep{EWW15}.   
    \end{enumerate}
    
Thus, there exist many settings when one can, and should, use e-values to quantify evidence against a null. (However, there of course remain innumerable situations in which p-values are perfectly reasonable choices.)

To reflect some of the above points, in a concrete example below, we will see that a complicated dependence structure appears both within experiment and across experiments, and one can easily construct useful e-values in such situations, with almost no model assumptions on the test statistics.  
\begin{example}  \label{ex:21}
Suppose there are $K$ traders (or machines), and a researcher is interested in knowing which ones are skillful (or useful). This is a classic problem in finance; see e.g., \cite{BSW10} in the context of detecting mutual fund skills. 
For $k=1,\dots,K$,
the null hypothesis $H_k$ is that trader $k$ is not skillful, 
meaning that  they make no profit on average (without loss of generality we can assume the market risk-free return rate is $0$).
The nonnegative random variables $X_{k,1},\dots,X_{k,n}$ are the monthly realized 
performance (i.e., the ratio of payoff to investment; $X_{k,j}>1$ presents a profit and $X_{k,j}<1$ means a loss)
 of agent $k$ from month $1$ to month $n$.
 The no-skill null hypothesis is $\E[X_{k,j} \mid \mathcal F_{j-1}] \le 1$ for $j=1,\dots,n$,
 where the $\sigma$-field
$\mathcal F_t$ represents the available market information up to time $t\in\{0,\dots,n\}$,
and we naturally assume that $(X_{k,t})_t$ is adapted to $(\mathcal F_t)_t$.
 
Since the agents are changing investment strategies over time 
and all strategies depend on the   financial market evolution,
there is complicated serial dependence within $(X_{k,1},\dots,X_{k,n})$
for single $k$, as well as cross dependence among agents $k=1,\dots,K$. 
Because of the complicated serial  dependence  and the lack of distributional assumptions of the performance data, it is difficult to obtain useful p-values for these agents.
Nevertheless, we can easily obtain useful e-values: for instance, $E_k=\prod_{j=1}^n X_{k,j}$ is a valid e-value, as well as any mixture of U-statistics of $X_{k,1},\dots,X_{k,n}$, including the mean and the product ($X_{k,1},\dots,X_{k,n}$ are called sequential e-values by \citep{VW20}, who also propose  methods to merge them into one e-value).
Moreover, the obtained e-values $E_1,\dots,E_K$ are dependent in a complicated way. Even if these e-values are not very large, they can be useful for other studies on these  traders. 
We return to a similar example in the simulations (Section~\ref{sec:simulation}). 

 As a side note, other sophisticated e-values can also be constructed for this problem using exponential self-normalized processes \citep{HRMS20}. For example, define $\psi(\lambda) := -\lambda - \log(1-\lambda)$. Then, $E_k = \prod_{j=1}^n \exp\left( \lambda (X_{k,j} - 1) - \psi(\lambda)(X_{k,j}-1)^2 \right)$ for for any $\lambda \in [0,1)$. Note that $\lim_{\lambda \to 0} \frac{\psi(\lambda)}{\lambda^2/2} = 1$, so this e-value captures subGaussian-like right tails of $\sum_{j=1}^n X_{k,j}-1$ relative to its empirical variance $\sum_{j=1}^n (X_{k,j}-1)^2$. To avoid picking a fixed $\lambda$, one can instead simply use a mixture over $\lambda$, for example using a Gamma distribution peaking near 0~\citep{HRMS21}. One can also pick $\lambda_j$ predictably at each step. These observations form the basis of all time-uniform Empirical-Bernstein inequalities \citep{HRMS21,WSR20a,WSR20b}.
\end{example}

To summarize, there are several reasons to work with e-values: they arise naturally in sequential settings, we know how to construct e-values in settings where we do not know how to construct p-values, and e-values can be more robust to misspecification or uncertain asymptotics in high-dimensional settings. Of course, there are also many reasons \emph{not} to work with e-values: in particular, p-values will often yield more ``powerful" (single/multiple) tests when the underlying modeling assumptions are true. Thus, e-values and p-values can also be viewed as picking different points on the validity-power curve: if one takes more risks with validity (p-values), one can gain more power, but if one questions one or more modeling assumptions, e-values may provide a safer choice.

\begin{remark}
Recall that we can always convert an e-value $e$ into a p-value $p=1/e$. One can also transform a p-value $p$ into an e-value $e=f(p)$ using ``calibrators''; recalling the definition by \citet{SSVV11}, a calibrator is a decreasing function $f:[0,1]\to[0,\infty)$ such that 
\begin{equation}\label{eq:calibrator}
    \int_0^1 f(u)\d u = 1.
\end{equation}
If the only known way to construct an e-value is by calibrating a bonafide p-value, then e-BH will typically not be more powerful than BH. Nevertheless, in some of our examples, we will assume that we can construct p-values, and calibrate them to e-values and study the resulting e-BH procedure, in order to understand the difference between the two approaches.
In contrast, if the only known way to construct a p-value is by inverting an e-value, then working with e-values directly is beneficial because the e-BH procedure is robust to arbitrary dependence without corrections, unlike the BH procedure.
\end{remark}

\section{Recap: the BH procedure}
\label{sec:2}

We briefly review the BH and BY
 procedures (\cite{BH95,BY01}) for FDR control, as the benchmark for our discussions.

To discuss the dependence structure among p-values and e-values,
we rely on  the notion of \emph{positive regression dependence on a subset (PRDS)} of \cite{BY01}, flipped when imposed on e-values. 
  A set $A\subseteq \R^K$ is said to be \emph{decreasing} (resp.~\emph{increasing})
if $\mathbf x\in A$ implies $\mathbf y\in A$ for all $\mathbf y\le \mathbf x$ (resp.~all $\mathbf y\ge \mathbf x$).   In this paper, all terms ``increasing" and ``decreasing" are in the non-strict sense,
and  inequalities should be
interpreted component-wise when applied to vectors.
  \begin{enumerate}
  \item  A random vector $\mathbf P$ of p-values is PRDS on $\mathcal N$ (or, simply, satisfies PRDS) if for any null index $k\in \mathcal N$ and   increasing set $A  \subseteq  \R^K$, the
function $x\mapsto \p(\mathbf P\in A\mid P_k\le x)$ is increasing on $ [0,1]$.  
  \item A random vector $\mathbf E$ of e-values satisfies PRDS if for any null index $k\in \mathcal N$ and   decreasing set $A  \subseteq  \R^K$, the
function $x\mapsto \p(\mathbf E\in A\mid E_k\ge x)$ is decreasing on $ [ 0,\infty)$. 
  \end{enumerate}
We chose to use the version of PRDS in \cite[Section 4]{finner2009false} (also see \cite{BR17}) which is weaker than the original one used in \cite{BY01}, but the difference is relatively minor:  ``$P_k \leq x$" in the former is replaced by ``$P_k=x$" in the latter.
   
If the null p-values (e-values) are mutually independent and independent of the non-null p-values (e-values), then PRDS holds; as such, PRDS is a generalization of independence. 
Moreover,  increasing individual transforms do not affect the PRDS property. Further, the PRDS property is preserved when moving from p-values to e-values using calibrators, or vice versa by inversion.  We record this fact below.
\begin{fact}
The PRDS property of p-values and that of e-values are equivalent in the following sense: 
If $\mathbf P$ is a vector of PRDS p-values, and $E_k = f_k(P_k)$ for any calibrator $f_k$ from \eqref{eq:calibrator}, then $\mathbf E$ is a vector of PRDS e-values.  
Similarly, if $\mathbf E$ is a vector of PRDS e-values, then $\mathbf P$ is a vector of PRDS p-values, where $P_k = 1/E_k$.
As an important example, PRDS holds if p-values/e-values are built on positively correlated (jointly) Gaussian test statistics $X_1,\dots,X_K$ (themselves PRDS).
For instance, this includes p-values obtained from Neyman-Pearson tests, i.e.,
$P_k=1-\Phi(X_k)$ where $\Phi$ is the standard Gaussian cdf, 
and e-values obtained from   mixture likelihood ratios,
 $
E_k = \int \exp({\delta X_k -\delta^2/2}) \d \nu_k (\delta),
 $
where $\nu_k$ is an arbitrary mixing distribution;
 see Example~\ref{ex:LR} for instance.
 The PRDS property of these p-values and e-values directly follows from the invariance of PRDS under monotone transforms. 
\end{fact}

Following the literature, we will study FDR in both the case of PRDS and that of arbitrary dependence.
The input of the Benjamini-Hochberg (BH) procedure includes three ingredients:
\begin{enumerate}[(a)]
\item  $K$ realized p-values $p_1,\dots,p_K$ associated to $H_1,\dots,H_K$, respectively;
\item an FDR level $\alpha \in (0,1)$;
\item (optional) dependence information or assumption on p-values, such as independence, PRDS or no information.
\end{enumerate}

For $k\in \mathcal K$, let $p_{(k)}$ be the $k$-th order statistics of $p_1,\dots,p_K$, from the smallest to the largest. 
The \emph{(base) BH procedure} $\cD(\alpha)$ rejects all hypotheses with the smallest $k^*$ p-values,
where 
\begin{equation}
\label{eq:p-k}
k^*=\max\left\{k\in \mathcal K: \frac{K p_{(k)}}{k} \le \alpha\right\},
\end{equation}
with the convention $\max(\varnothing)=0$, 
and accepts the rest.
We summarize some known results of \cite{BH95} and \cite{BY01} on the   procedure $\cD(\alpha)$ below.   Throughout, we write  
\[
\ell_K:=\sum_{k=1}^K \frac 1k \approx \log K.
 \] 
  \begin{theorem}\label{th:BH}
  For arbitrary p-values  and $\alpha\in (0,1)$, the base BH procedure $\cD(\alpha)$ satisfies
    $$
\E\left[\frac{F_{\cD(\alpha)}}{R_{\cD(\alpha)} }\right] \le \frac{ \ell_K K_0}{K}\alpha \le  \ell_K  \alpha.
$$
Moreover,   if   the p-values satisfy PRDS, then 
 $$
\E\left[\frac{F_{\cD(\alpha)}}{R_{\cD(\alpha)} }\right] \le \frac{K_0}{K}\alpha \le \alpha,
 $$
and the first inequality is an equality if the null p-values are iid uniform on $[0,1]$.
 \end{theorem}   
 
 By Theorem~\ref{th:BH}, if the p-values are PRDS, then the base BH procedure has an FDR of at most $\alpha$. 
If the p-values are arbitrarily dependent, then we need to replace $\alpha$ in the base BH procedure by $\alpha/\ell_K$, resulting in 
the Benjamini-Yekutieli (BY) procedure $\cD'(\alpha):=\cD(\alpha/\ell_K)$,
which  rejects all hypotheses with the smallest $ k^*$ p-values,
where \begin{equation}\label{eq:BY1}k^*=\max\left\{k\in \mathcal K: \frac{\ell_K K p_{(k)}}{k} \le \alpha\right\}.\end{equation}   
The extra factor $\ell_K$ in the BY procedure reflects the fact that it is harder to justify discoveries without the PRDS assumption, which includes independence as a special case. 
 It is well known that this factor cannot be   improved in general.

\section{The e-BH procedure}
\label{sec:4}

In this section, we design an analog of the BH procedure for e-values, which we call the \emph{e-BH procedure}.
Similar to the BH procedure, the input of the e-BH procedure 
includes three ingredients:
\begin{enumerate}[(a)]
\item  $K$ realized e-values $e_1,\dots,e_K$ associated to $H_1,\dots,H_K$, respectively;
\item an FDR level $\alpha \in (0,1)$;
\item (optional) distributional information or assumption on e-values.
\end{enumerate}
While (a) and (b) are the same as those of the BH procedure, (c) is somewhat different: in addition to dependence information, e-BH can also accommodate information of marginal distributions, since e-values have more freedom than p-values in terms of their distributions (the former are only constrained by their expectations under the null). 

The e-BH procedure can be described in two simple steps.
  \begin{enumerate}
 \item (Optional) Boost the \emph{raw e-values} in (a) using information in (c).  
 \item Apply the \emph{base e-BH procedure} to the boosted e-values and level $\alpha$.
 \end{enumerate}
We  will first describe the base e-BH procedure
in step 2
and then explain how to boost e-values in step 1.
If there is no available information in (c), we can simply skip step 1 above, and 
directly apply the base e-BH procedure to the raw e-values, which always has a valid FDR control at level $\alpha$.

\subsection{The base e-BH procedure}

Let $e'_1,\dots,e'_K$ be the boosted e-values obtained from step 1 of the procedure.
Keep in mind that they can be chosen as identical to the raw e-values $e_1,\dots,e_K$ if there is no information in (c) or one opts to skip the boosting in step 1.     

The base e-BH procedure is   applying the BH procedure to $1/e'_1,\dots,1/e'_K$.
To be precise, for $k\in \mathcal K$, let $e'_{[k]}$ be the $k$-th order statistic of $e'_1,\dots,e'_K$, sorted from {the largest to the smallest} so that $e'_{[1]}$ is the largest boosted e-value. 
The base e-BH procedure $\cG(\alpha):[0,\infty]^K\to 2^{\mathcal K}$ then rejects hypotheses with the largest $k_e^*$ (boosted) e-values,
where 
\begin{equation} 
\label{eq:e-k} 
k_e^*=\max\left\{k\in \mathcal K: \frac{k e'_{[k]}}{K} \ge \frac{1}{\alpha}\right\}.
\end{equation}   

We note that, analogous to the BH procedure, there are many equivalent ways of describing the e-BH procedure, including using an estimate of the false discovery proportion as pioneered by \cite{S02}, or the self-consistency viewpoint of \citet{BR08}. 
The alternative formulation in the next proposition  will be helpful when showing several results on FDR of the e-BH procedure. 
 \begin{proposition}
\label{prop:r1-1}
Let $e_1',\dots,e_k'$ be the raw or boosted e-values. Define
  \begin{align}
  \label{eq:r1-talpha}
  R(t):= |\{k\in \mathcal K: e_k'\ge t\}| \vee 1 \mbox{~~~and~~~}
 t_\alpha: = \inf \{t\in [0,\infty):t{R(t)} \ge K/\alpha\}.
\end{align}
For each $k$, the base e-BH procedure applied to $e_1',\dots,e_K'$ and level $\alpha$ rejects 
 $H_k$ if and only if $e'_k\ge t_\alpha$.  Moreover, $t_\alpha R(t_\alpha)= K/\alpha$.
 \end{proposition}
 
A self-contained simple proof of Proposition \ref{prop:r1-1} is put in  
 Appendix \ref{app:pf}.

\subsection{Boosting e-values} \label{sec:42}

 To boost e-values in step 1, there are two types of information we can use: marginal distributional information and joint dependence information. Below we first describe our standard choice of boosting, and then the more general choices which can be freely chosen by the user (Remark~\ref{rem:32}).
 Although we generally recommend the standard choice, the flexibility of choosing boosting functions allows us to include the p-based methods (like BH and BY) as special cases.
   
Let $K/\mathcal K :=\{K/k:k\in \mathcal K\}$, and  define a truncation function $T:[0,\infty]\to [0,K]$ by
letting $T(x)$ be the largest number in $K/\mathcal K\cup\{0\}$ that is no larger than $x$. In other words, 
\begin{equation} \label{eq:def-T}
T(x)  = \frac{K}{\lceil  K/x\rceil}\id_{\{x\ge 1\}}\mbox{~~with $T(\infty)=K$}.
\end{equation} 
Note that $T$ truncates $x$ to take only values in $K/\mathcal K$ (or zero).
The truncation function $T$ will be frequently used in this paper. 
The standard boosting method is now described below.   
For each $k\in \mathcal K$, 
take a \emph{boosting factor} $b_k\ge 1$ (the larger, the better)  such that, depending on dependence information,
 \begin{align}
 \label{eq:enhance-1}    \max_{x\in K/\mathcal K}  x  {\p(\alpha b_k E_k \ge x)} 
 \le \alpha   & \mbox{~~~~~~if e-values are PRDS;}\\ 
 \E[T(\alpha b_k E_k)]\le \alpha    & \mbox{~~~~~~otherwise,}
    \label{eq:enhance-2}  
 \end{align} 
  where the expectation and the probability are computed under the null distribution of $E_k$, which is sometimes known. 
     In case of a composite null,
 \eqref{eq:enhance-1}  and  \eqref{eq:enhance-2} need to hold for all probability measures in $H_k$, so an additional supremum over $H_k$ must be appended to the left hand side.  
By Markov's inequality, $b_k$ in \eqref{eq:enhance-1} is larger or equal to that in \eqref{eq:enhance-2}, leading to a stronger boosting under PRDS.
 Since $\E[E_k]\le 1$ for a null e-value $E_k$, \eqref{eq:enhance-1} and \eqref{eq:enhance-2} always hold  for $b_k=1$, making $b_k=1$ always a safe choice even if there is no additional (marginal or joint) distributional information. Moreover, the boosting factor $b_k$ is equal to $1$  if no marginal information on the distributions of the e-values is available. 
In either case, define the boosted e-values 
\[
e_k'=b_ke_k
\] 
for $k\in \mathcal K$. On first reading, the reader may skip the next two remarks.

\begin{remark}
Note that the left-hand sides of \eqref{eq:enhance-1} and \eqref{eq:enhance-2}
are increasing in $b_k$, and hence a practical value of $b_k$ can be obtained by simply trying a few choices of $b_k\ge 1$. It would be ideal to find the largest  $b_k$ such that \eqref{eq:enhance-1} or \eqref{eq:enhance-2} becomes an equality. This is possible in some cases  but it may not be possible if the null distribution of $E_k$ is not continuous or not completely specified; see Examples~\ref{ex:calibrator},~\ref{ex:LR},~\ref{ex:calibratorPRDS} and~\ref{ex:LR2}, where we either get a precise value  or an analytical approximation of the best $b_k$.
Under an extra assumption, we have a simple formula (Proposition~\ref{prop:bk}) for the best $b_k$ in case e-values are PRDS. 
\end{remark}

\begin{remark}\label{rem:32}
Instead of using the standard method in \eqref{eq:enhance-1} and \eqref{eq:enhance-2} described above, the user can instead choose increasing   functions $\phi_1,\dots,\phi_K:[0,\infty]\to [0,\infty]$ such that 
 \begin{align}
 \label{eq:enhance-1p}    \max_{x\in K/\mathcal K}  x  {\p(\alpha \phi_k (E_k) \ge x)} 
 \le \alpha   & \mbox{~~~~~~if e-values are PRDS;}\\ 
 \E[T(\alpha \phi_k(E_k))]\le \alpha    & \mbox{~~~~~~otherwise,}
    \label{eq:enhance-2p}  
 \end{align} 
and define the boosted e-values $e_k'=\phi_k(e_k)$ for $k\in \mathcal K$.
The choice $\phi_k:x\mapsto b_k x$ corresponds to the standard method  explained earlier. Note that if e-values are PRDS, then so are the boosted e-values by any choice of boosting, since PRDS is invariant under monotone transforms.  
\end{remark}

\section{The FDR guarantee of the e-BH procedure}

 We are now ready to state our main result.  
 
 \begin{theorem}\label{th:validity}
 The e-BH procedure at level $\alpha$ has FDR at most $K_0\alpha/K$. 
 In particular, the base e-BH procedure  applied to arbitrarily dependent raw e-values has FDR at most $K_0\alpha/K$.
 \end{theorem} 
The full proof of Theorem~\ref{th:validity} follows from Theorems~\ref{th:e2} and~\ref{th:e1}  in the next few sections, which requires   delicate technical treatment. Here, we provide a simple proof for the second statement of Theorem~\ref{th:validity} 
that the base e-BH procedure has the desired FDR guarantee. This simple argument illustrates the advantages of working with e-values, and it applies to  any ``complaint'' e-testing procedures including the base e-BH procedure.  
Following~\cite{BR08}, who defined self-consistent p-testing procedures, an e-testing procedure $\cG$  is said to be \emph{self-consistent at level $\alpha\in(0,1)$} if, denoting by $R_{\cG}$ the number of rejections, every rejected e-value $e_k$ satisfies
$$
e_k \ge \frac{K}{\alpha R_{\cG}  }.
$$ 
Using $1/p_k$ in place of $e_k$ above, we recover  the definition from~\cite{BR08} with a linear shape function. 
Clearly, the base e-BH procedure is self-consistent because of \eqref{eq:e-k}; moreover, the base e-BH procedure dominates all other self-consistent e-testing procedures by definition.

\begin{proposition}\label{th:compliant}
Any self-consistent e-testing procedure  at level $\alpha $ has FDR at most $\alpha K_0/K$ for arbitrary configurations of e-values.
\end{proposition}
\begin{proof} 
Let $\cG$ be a self-consistent e-testing procedure.  
Let $\mathbf E=(E_1,\dots,E_K)$ be an arbitrary vector of e-variables fed to the testing procedure $\cG$. 
The FDP of $\cG$ satisfies 
\begin{align} 
 \frac{F_{\cG}}{R_{\cG} }   =
 \frac{ |\cG (\mathbf E)\cap \mathcal N |  }{R_{\cG} \vee 1}  
 & =\sum_{k\in \mathcal N}  \frac{ \id_{\{k\in \cG(\mathbf E)\}}  }{R_{\cG} \vee 1}  
  \le \sum_{k\in \mathcal N}  \frac{ \id_{\{k\in \cG(\mathbf E)\}} \alpha E_k   }{K}   \le \sum_{k\in \mathcal N}   \frac{\alpha   E_k  }{K} , 
\label{eq:compliant}
\end{align} 
where the first inequality  is due to self-consistency. 
As $\E[E_k]\le 1$ for $k\in \mathcal N$, we have
$$
\E\left[ \frac{F_{\cG}}{R_{\cG} }  \right]  \le   \sum_{k\in \mathcal N}  \E\left[ \frac{\alpha   E_k  }{K}   \right] \le \frac{\alpha K_0}{K},
$$
thus the desired FDR guarantee. 
\end{proof}
The above proof can be alternately viewed as verifying the ``dependence condition'' in \cite{BR08}, which holds even for arbitrarily dependent e-values unlike for p-values.\footnote{We thank a reviewer for this insightful observation.}

In contrast, general self-consistent p-testing procedures do not have  the nice property  in Proposition~\ref{th:compliant} even if the p-values are independent. Indeed, under a condition slightly weaker than PRDS, \cite{Su18} proves that a self-consistent p-testing procedure\footnote{\cite{Su18} used the term \emph{compliance} for self-consistency. We thank a reviewer for pointing out the terminology of self-consistency. 
\cite{BR08} showed that a self-consistent p-testing procedure that is decreasing in each p-value controls FDR at level $\alpha$ for independent or PRDS p-values. 
} has a weaker FDR guarantee 
$
\alpha (1+\log(1/\alpha) )> \alpha.
$
  Hence, for an FDR guarantee of $\alpha$, the p-testing procedure needs to be run at level $\alpha'<\alpha$ satisfying
\begin{align}\label{eq:alphaprime}
\alpha' \left(1+\log \frac{1}{\alpha'}\right) = \alpha.
\end{align}
 
The FDR control inflation to $\alpha (1+\log(1/\alpha))$ for arbitrary self-consistent p-testing procedures is quite a contrast to the control at level $\alpha$ for arbitrary self-consistent e-testing procedures, and this fact plays an important role in bandit multiple testing~\citep{JJ18,xu2021unified}. 
 
\subsection{Self-consistent structured, post-selection and grouped e-BH}

Without further structure, self-consistent p- or e-testing procedures are dominated by the base BH or e-BH procedure, and hence we do not directly apply them for a stand-alone data set of p-values or e-values. Nevertheless, they become useful in multi-armed bandit problems; see \cite{JJ18} and our Section~\ref{sec:simulation}. Self-consistent procedures are also useful in structured settings. Suppose we would like to insist that the rejected set cannot be an arbitrary subset of $\mathcal K$, but must be one amongst a given class of sets $\mathcal S \subseteq 2^{\mathcal K}$. For example, if the hypotheses are structured as a graph, then we may choose $\mathcal S$ to be all connected subgraphs. The appropriate modification of the e-BH procedure (``structured e-BH procedure'') would try to 
 \begin{quote}
     find any set $S \in \mathcal S$ such that every e-value in $S$ is larger than $\tfrac{K}{\alpha |S|}$.
\end{quote}    
    Importantly, solving for the \emph{largest} $S$ may not be computationally feasible, so one may wish to use an approximation algorithm (like a greedy heuristic that grows $S$ by starting with the largest e-value and adding a few elements at a time to this set) to find \emph{some large} $S \in \mathcal S$ such that every e-value in $S$ is larger than $\tfrac{K}{\alpha |S|}$. Such a procedure would immediately satisfy an FDR guarantee, for any $\mathcal S$ and dependence structure, by virtue of being self-consistent. In contrast, the ``structured BH procedure'' in \cite[Section 8]{RBWJ19} requires the BY reshaping correction in order to maintain validity.

Another application of self-consistency is in post-processing e-values that have been screened or filtered based on the data. Formally, suppose the original set of hypotheses $\mathcal K$ has been shrunk to a data-dependent subset $S \subset \mathcal K$ in some arbitrary way based on the e-values. This could be based on any subjective choices of the user (based on side-information, prior scientific knowledge, whatever). The question is: how do we now find a subset of $S$ that guarantees FDR control? The ``post-selection e-BH procedure'' is simple: 
\begin{quote}
    run the e-BH procedure on $S$ at the amended level $\widetilde \alpha:= \alpha |S|/K$. 
\end{quote}
Then every rejected e-value will be larger than $\tfrac{|S|}{\widetilde \alpha R_{\cG}}$, which equals $\tfrac{K}{\alpha R_{\cG}}$. Thus, this procedure controls FDR under arbitrary dependence for any initial selection procedure. This is once more in contrast to the ``post-selection BH procedure'' in \cite[Section 8]{RBWJ19}, where conditions on both the dependence and the selection are needed for FDR control, devoid of which one must use the BY reshaping correction.

Combining the above two ideas (imposition of structural constraints and screening/filtering) in any order immediately lends a great deal of flexibility to the user to design new procedures. For example, it immediately yields a ``focused e-BH'' procedure, an extension of the recent focused BH algorithm by~\cite{KSB18}, whose details are omitted for brevity. In contrast to \citep[Theorem 1]{KSB18}, our procedure controls FDR under no assumptions whatsoever on the dependence, structure and filtering. 

As a final extension, imagine that we are given a partition of the hypotheses into $G$ prespecified groups. It is easy to define group-level e-values by merging the individual level e-values (e.g., averaging under arbitrary dependence, or multiplying under independence). This immediately yields a group-level e-BH procedure that can be used as an initial filtering step before running the post-selection e-BH procedure. For consistently aggregating discoveries across multiple partitions, an e-filter extension of the p-filter algorithm~\citep{RBWJ19} may be derived.

\subsection{A few remarks on the e-BH procedure}\label{sec:52}
At this point, we have seen that the base e-BH procedure has the desirable FDR control. Let us make a few remarks on the (full) e-BH procedure to further explain its features and subtleties. Some features have been mentioned briefly before and here  they are discussed in comparison with alternatives. 
\begin{enumerate}
\item Although one can always opt to skip the boosting step and directly apply the base e-BH method to raw e-values, 
boosting can be quite helpful to enhance detection power. The boosting factor $b_k$ may be  substantial (see Examples~\ref{ex:calibrator},~\ref{ex:LR},~\ref{ex:calibratorPRDS} and~\ref{ex:LR2}); a boosting factor between 1.3 and 10 is common in stylized settings.  

\item  
If the e-values are independent (or   PRDS), then by Theorem~\ref{th:BH}, the FDR of the base e-BH procedure is at most $\alpha$.
This result is not of primary interest, but is worthy of note.
The main advantage of e-BH is that, unlike the BH procedure, the above FDR guarantee holds under arbitrary dependence, where a conversion to p-values would be suboptimal due to the extra $\approx \log K$ correction factor. This point is relevant in case p-values are derived from e-values, for example via universal inference \citep{WRB20}.

\item The boosting in step 1 of the BH procedure can easily incorporate partial information on the null  distributions of the e-values. 
If we  know a set of null possible distributions of $E_k$, then it suffices to calculate a boosting factor under each distribution, and take their infimum. Moreover, if only the null distributions of some e-values are known, we can simply choose $b_k=1$ for the ones whose distributions we do not know. 
Note that the post-selection and structured e-BH procedures also enjoy FDR control after boosting under arbitrary dependence.

\item Since $b_k$ may be different for each $k$, the boosted e-values may not have the same order as the raw e-values. Hence, the full e-BH procedure may reject a hypothesis with a smaller e-value while accepting one with a larger e-value.
 This feature is intentional. For instance, an observation of $e_k=1.999$ 
 carries stronger   evidence against the null hypothesis of a uniform distribution on $[0,2]$, compared to an observation of $e_k=2$ against a null hypothesis of an exponential distribution with mean $1$ or an unknown distribution.

\item E-values typically contain less information than p-values (the flip side of needing less assumptions/structure/knowledge/modeling in order to be constructed), and hence a simplistic comparison with the BH procedure applied to p-values is not particularly insightful.
Nevertheless, if the  distributions of null e-values are fully known, then the e-BH procedure performs comparably to the BH and the BY procedures; see Section~\ref{sec:gen-p} for their connection and Section~\ref{sec:simulation} and Appendix~\ref{app:num} for simulation results. 

\item  We note a contrasting feature on the FDR of the BH and e-BH procedures.
The BH procedure has a desirable FDR control under dependence assumptions (such as independence and PRDS) 
and a penalty if such assumption is not available (e.g., \citep{BY01,FL20}). 
In contrast, the e-BH procedure has a desirable FDR control with no dependence assumption and a boosting of power is possible if dependence assumptions are imposed. 
Similarly, a penalty in \eqref{eq:alphaprime} needs to be applied for self-consistent p-testing procedures, which are dominated by BH, but not for self-consistent e-testing procedures. 

\item Suppose we may avail ourselves of only approximate e-values that (under the null) satisfy $\mathbb{E}[E_i] \leq 1 + \epsilon_i$, or ``asymptotic" e-values, which satisfy $\mathbb{E}[E_i] \leq 1 + o(1)$, in the sense that as the amount of data used to calculate the e-value increases in size, then the approximate e-value becomes a valid e-value in the limit. Then the corresponding FDR control in the first case is simply bounded by $\tfrac{\alpha}{K}(K_0+ \sum_{i \in \mathcal{N}} \epsilon_i)$, which is further bounded by the simpler expression $\alpha(1+\max_i \epsilon_i)$. In the second case, the FDR bound becomes $\alpha (1 + o(1)) K_0/K$. In other words, errors in the e-values (which may be due to slightly violated assumptions, or nuisance parameters, etc.) directly and linearly propagate to errors in the FDR control level. This claim can be verified by simply observing the linearity of expectation used in the last line of the proof of Proposition~\ref{th:compliant}. We know of no existing simple, analogous claim for the BH procedure or other p-testing procedures.

 \end{enumerate}

In practice, it may be desirable to assign weights based on prior knowledge of each hypothesis. Here, we describe an analog of the weighted BH procedure~\citep{BH97}. 
For arbitrary constant weights $w_1,\dots,w_K$ summing up to $K$,
a similar boosting scheme is obtained by each
 replacing $\alpha$ with $w_k\alpha$ in \eqref{eq:enhance-1} and \eqref{eq:enhance-2}.
We can alternatively skip the boosting in step 1 and directly use $w_1E_1,\dots,w_KE_K$ as the input e-values fed to the base e-BH procedure. 
In either case,  the full e-BH procedure still has the valid FDR control at level $\alpha$, under arbitrary dependence of the e-values as before.  
This claim is also justified by Theorems~\ref{th:e2} and~\ref{th:e1} (see the explanation after  Theorem~\ref{th:e2}).

 At this point, most (but not all) messages for the practitioner have been delivered and the casual reader may be warned that the paper intentionally takes a rather theoretical turn, and analyzes the e-BH procedure more carefully, comparing repeatedly to the BH procedure or to the case when e-values are simply calibrated p-values. Proofs of all results beyond this section are put in Appendix \ref{app:pf}.

 \section{Theoretical analysis}

\subsection{A technical  lemma} 
While we have provided a direct  and simple proof for FDR control of the base e-BH that does not require any correction under arbitrary dependence, the reader may be curious to understand better why exactly BH has to pay an extra $\approx \log K$ factor, but e-BH, on the contrary, allows for boosting. We provide one technical answer here in the form of a  lemma, similar to \cite[Lemma 1]{RBWJ19} in the setting of p-values.  
Lemma~\ref{lem:e-FDR} leads to FDR control of the full e-BH procedure applied to boosted e-values.

Recall that for $k\in \mathcal K$, a random vector $\mathbf X=(X_1,\dots,X_K)$   is PRDS on $X_k$ if for any   decreasing set $A  \subseteq  \R^K$, the
function $x\mapsto \p(\mathbf X
\in A\mid X_k\ge x)$ is decreasing on $ [ 0,\infty)$.   

In order to prepare for the lemma that follows, note that for any positive constant $c$, we have that  
\begin{equation}\label{eq:setup-e}
\mathbb E[c \id_{\{X_k \geq c\}}] \leq \mathbb E[X_k],
\end{equation}
which is simply a restatement of Markov's inequality. The following lemma addresses the situation where $c$ is a random data-dependent quantity $f(\mathbf X)$,  as in the computation of FDR of the e-BH procedure.

 \begin{lemma} \label{lem:e-FDR}
Take an arbitrary random vector $\mathbf X=(X_1,\dots,X_K) :\Omega\to [0,\infty]^K$ and fix $k\in \mathcal K$. Let $f:[0,\infty]^K\to  [0,\infty)$ be a measurable  function with range $I_f$,
and $X_k':= \sup \{x\in I_f \cup\{0\}:  x\le X_k\}$ so that $X_k' \le X_k$ by construction.
\begin{enumerate}[(i)]
\item  If $X_k$ is independent of $\mathbf X^{-k}:=(X_j)_{j\ne k}$, then 
$$
\E\left[f(\mathbf X) \id_{\{X_k\ge f(\mathbf X)\}}  \mid \mathbf X^{-k} \right] \le \E[X_k' ] \le \E[X_k].
$$
\item   If  $f$ is decreasing and $\mathbf X$ satisfies  PRDS on $X_k$, then 
$$
\E\left[f(\mathbf X) \id_{\{X_k\ge f(\mathbf X)\}} \right] \le \sup_{x \ge 0} x \p(X'_k\ge x) \le \E[X_k].
$$
\item For any dependence structure, it always holds that
$$
\E\left[f(\mathbf X) \id_{\{X_k\ge f(\mathbf X)\}}   \right] \le \E[ X_k'] \le \E[X_k].
$$
\end{enumerate} 
In particular, if $X_k$ is a null e-value,   all expectations above are bounded by $1$.
 \end{lemma}

 Later we will apply Lemma~\ref{lem:e-FDR} to  $\mathbf X=\mathbf E'$ 
 and a specific choice of $f$ to show Theorems~\ref{th:e2} and~\ref{th:e1}.

Comparing Lemma~\ref{lem:e-FDR}  to \cite[Lemma 1]{RBWJ19} for the p-value setting, we observe that the first two statements have a very direct parallel. However the third statements are where the difference appears; the p-value statement requires a correction while the above e-value statement does not. To briefly elaborate, recall~\eqref{eq:setup-e} and correspondingly note that a null p-variable $P_k$ satisfies for any constant $c \in (0,1]$,
\begin{equation}\label{eq:setup-p}
\E\left[ \frac{\id_{\{P_k \leq c\}}}{c}\right] \leq 1.
\end{equation}
While  Lemma~\ref{lem:e-FDR} showed that \eqref{eq:setup-e} easily generalized from a constant $c$ to a data-dependent $f(\mathbf X)$,
\cite[Lemma 1]{RBWJ19} showed that a corresponding statement for \eqref{eq:setup-p} is true for p-variables only under a PRDS assumption, but otherwise an extra $\ell_K$ factor is paid under arbitrary dependence.

We anticipate this lemma to aid with the design and proof of FDR procedures with e-values in other contexts---just as the corresponding p-value lemma already has~\citep{RCJW19}---and we shall see how it is applied to e-BH next.

\subsection{FDR of e-BH for arbitrary e-values}
\label{sec:eBH-arbitrary}

We   analyze some  properties of the base e-BH procedure.
Below, the  e-values we feed into the base e-BH procedure 
are the boosted e-values from step 1 of the BH procedure.
We will omit the term ``boosted" while keeping in mind that the base e-BH can also be directly applied to the raw e-values. 
Let $(E'_1,\dots,E'_K)$ be the random vector of (boosted) e-values and $(e'_1,\dots,e'_K)$ be its realized value. 
The only property we will use  on $ E'_1,\dots,E'_K $ is that they are nonnegative. 
The next result is a more  precise analysis of the FDR of the base e-BH procedure   under arbitrary dependence. 

\begin{theorem}\label{th:e2}
  Applied to arbitrary non-negative random variables $ E'_1,\dots,E'_K $ and $\alpha\in (0,1)$, the base e-BH procedure $\cG(\alpha)$ satisfies 
  $$
\E\left[\frac{F_{\cG(\alpha)}}{R_{\cG(\alpha)} }\right] = \frac{\alpha}{ K}
\sum_{k\in \mathcal N}  \E\left[ t_\alpha \id_{\{E'_k\ge t_\alpha\}} \right] \le \frac{  K_0}{K} y_\alpha,  
$$
where  $t_\alpha$ is given by \eqref{eq:r1-talpha}, and
 \begin{equation}\label{eq:yalpha}y_\alpha=\frac 1 {   K_0} \sum_{k\in \mathcal N}\E\left[   T(\alpha E'_k)\right].  \end{equation}  
In particular, if $E'_1,\dots,E'_K$ are the raw e-values or the boosted e-values via \eqref{eq:enhance-2} or \eqref{eq:enhance-2p}, then $y_\alpha\le \alpha.$
\end{theorem}

 Let $(E_1,\dots,E_K)$ be the vector of raw e-values.
  Theorem~\ref{th:e2} shows that for any choice of $b_1 E_1,\dots,b_KE_K$ satisfying \eqref{eq:enhance-2} or $\phi_1(E_1),\dots,\phi_K(E_K)$ satisfying 
\eqref{eq:enhance-2p}, the FDR of the e-BH procedure is at most $  \alpha$ under arbitrary dependence. 
Moreover, the FDR control of the weighted e-BH follows from  
$$
 \frac{  K_0}{K} y_\alpha 
 =
 \frac 1 {   K } \sum_{k\in \mathcal N}\E\left[   T(\alpha E'_k)\right]
 \le \frac{1}{K}\sum_{k\in \mathcal N} w_k \alpha \le \frac{1}{K}\sum_{k\in \mathcal K} w_k \alpha =\alpha.
$$

The value of $y_\alpha$ depends on the distribution of the boosted e-values as well as the number $K$ of hypotheses.
In case the null distributions of the e-values are known, 
we would ideally set $y_\alpha$ to $\alpha$ (or close to $\alpha$) by properly choosing $b_k$ in \eqref{eq:enhance-2} or $\phi_k$ in \eqref{eq:enhance-2p}.

In case  $\E[T(\alpha b_k E_k)]$ is not easy to compute,
it might be convenient to use a weaker  bound
\begin{equation}\label{eq:yalpha2}
\bar y_{\alpha,k}(b_k):=  \E[\alpha  b_k E_k \id_{\{\alpha b_k E_k\ge 1 \}}] \ge \E[T(\alpha b_k E_k)],
\end{equation}
and set $\bar y_{\alpha,k}(b_k)\le \alpha$ (ideally an equality) by choosing $b_k\ge 1$.
Such a choice of $b_k$ always guarantees the FDR of the e-BH procedure to be at most $\alpha$ by Theorem~\ref{th:e2},
since 
$$\frac 1 {K_0}\sum_{k\in \mathcal N} \E[T(\alpha b_k E_k)] \le \frac 1 {K_0}\sum_{k\in \mathcal N} \bar y_{\alpha,k}(b_k)\le \alpha.$$ 
In addition to being easier to compute, an advantage of 
$  \bar y_{\alpha,k}(b_k)$ is that it depends purely on the distribution of $E_k$ and not on $K$,
and hence the boosted e-value $b_k E_k$ is ready to use for other experiments involving $H_k$.

\begin{example}\label{ex:calibrator}
We illustrate  $y_\alpha$ and $\bar y_{\alpha,k}(b_k)$ with a popular class of calibrators in \cite{S20} and \cite{VW20}.
Take $\lambda \in (0,1)$ and 
assume  the raw e-values are given by, for $k\in \mathcal K$,
\begin{equation}\label{eq:ex-2} 
E_k = \lambda P_k^{\lambda-1},
\end{equation}
where $P_k$ is a uniform random variable on $[0,1]$  if $k\in \mathcal N$.  
We will consider the boosted e-values $b E_1,\dots,b E_K$ where $b$ is a common boosting factor since the null e-values are identically distributed. 
In this case, $$y_\alpha= \E[T(\alpha b E_1)] = \int_{0}^1 T\left(\alpha b \lambda u^{\lambda-1}  \right)\d u .$$
A formula for $\bar y_{\alpha,k}(b)$ is simple:  
$$ \bar y_{\alpha,k}(b) =\alpha b \int_{0}^{(\alpha b \lambda)^{1/(1-\lambda)}} \lambda u^{\lambda-1} \d u   = (  \lambda^\lambda \alpha b)^{1/(1-\lambda)}.$$
For instance, if $\lambda =1/2$, then  $\bar y_{\alpha, k} (b)=(\alpha b)^2/2$.  
Setting $\bar y_{\alpha, k} (b)=\alpha $ yields $b=(2/\alpha)^{1/2}$.
In this example, all e-values are boosted by a multiplier of $(2/\alpha)^{1/2}$, which is substantial;  e.g.,  $b\approx 6.32$ if $\alpha=0.05$. Finally, note that the same boosting factor $b$ above is also valid whenever the p-values are not exactly uniform but instead stochastically larger than uniform.
\end{example}

\begin{example}\label{ex:LR}
We consider e-values obtained from the likelihood ratio between two normal distributions with different mean and variance $1$ as used in the numerical experiment of  \cite{VW20}.
Take $\delta >0$ which represents the difference in the alternative and the null means, and 
assume  the raw e-values are given by, for $k\in \mathcal K$,
\begin{equation}\label{eq:ex-LR} 
E_k = e^{\delta X_k -\delta^2/2}
\end{equation}
where $X_k$ is a standard normal random variable if $k\in \mathcal N$.  
Note that each null e-value is a log-normal random variable with parameter  $(-\delta^2/2,\delta)$.
For the boosted e-values $b E_1,\dots,b E_K$, 
We have  
$$ \bar y_{\alpha,k}(b) = \alpha b \E[ E_k \id_{\{E_k \ge 1/(\alpha b)\}} ]= \alpha b \Phi \left(\frac{\delta}{2}+ \frac{\log (\alpha b)}{  \delta }\right),$$
where $\Phi$ is the standard Gaussian cdf.
Setting $\bar y_{\alpha, k} (b)=\alpha $ yields the equation
$$
  b  \Phi \left(\frac{\delta}{2}+ \frac{\log (\alpha b)}{  \delta }\right) =1,
$$
which can be easily solved numerically. 
For instance, if $\delta =3$ and $\alpha=0.05$, then  $b\approx 1.37$ by solving $\bar y_{\alpha, k} (b)=\alpha $,
and if $\delta=4$ and $\alpha=0.05$, then  $b\approx 1.11$.
These choices of $b$ are all  conservative  since $ \bar y_{\alpha,k}(b) $ is a conservative bound for $y_\alpha$ (but $b$ is only slightly conservative if $K$ is  large).  
\end{example}

For raw e-values $(E_1,\dots,E_K)$ with unspecified distributions, Theorem~\ref{th:e2} gives an upper bound $\alpha K_0/K$ on the FDR of the base e-BH procedure.
This upper bound is usually quite loose since $y_\alpha$  in \eqref{eq:yalpha} is typically much smaller than $\alpha$; see e.g., Example~\ref{ex:calibrator}.
This upper bound cannot be improved in general without any additional information.

\begin{example}[Sharpness of the upper FDR bound in Theorem~\ref{th:e2}]
\label{ex:43}
Consider the following setup. Let the $K_0$ null raw e-values be given by $E_k= K/(K_0 \alpha) \id_A$, $k\in \mathcal N$,
where $A$ is an event with $\p(A) = \alpha K_0/K$. Moreover, we set all other e-values to   $0$, which means that they cannot be rejected. Hence, the false discovery proportion is $1$ as soon as there is any rejection.
It follows that  
$$ \mathrm{FDR}_{\cG(\alpha)}  = \p(R_{\cG(\alpha)}>0) =
\p(A)= \alpha K_0/K,
$$
which is the upper bound provided by Theorem~\ref{th:e2}.
\end{example}

\subsection{FDR of e-BH  for PRDS e-values}
\label{sec:eBH-PRDS}

In case p-values are independent or positively dependent, 
the BH procedure has a lower FDR guarantee than the one obtained with arbitrary dependence. 

We investigate a similar matter for  e-values, aiming for a bound  better than that of Theorem~\ref{th:e2}. 
This is possible via the result of Lemma~\ref{lem:e-FDR}. 
To summarize, if e-values are PRDS, 
the base e-BH procedure has a smaller FDR guarantee, and, consequently, the e-BH procedure allows a more powerful boosting in its step 1.

\begin{theorem}\label{th:e1}
Suppose that the raw e-values are PRDS.
Applied to arbitrary non-negative random variables $E'_1,\dots,E'_K$ and $\alpha\in (0,1)$,   the  base e-BH procedure $\cG(\alpha)$ satisfies  
\begin{equation}\label{eq:the1}
\E\left[\frac{F_{\cG(\alpha)}}{R_{\cG(\alpha)} }\right] \le  \frac{K_0}{K }z_\alpha   .
\end{equation} 
where 
 \begin{equation}\label{eq:zalpha}z_\alpha=\frac 1 {   K_0} \sum_{k\in \mathcal N}\max_{x\in K/\mathcal K}  x  {\p(\alpha E'_k \ge x)}.  \end{equation}  
In particular, if $E'_1,\dots,E'_K$ are the raw e-values or the boosted e-values via \eqref{eq:enhance-1} or \eqref{eq:enhance-1p}, then $z_\alpha\le \alpha.$
\end{theorem}

 Theorem~\ref{th:e1}  shows that, under the assumption of PRDS,  using the boosted e-values $b_1E_1,\dots,b_KE_K$ satisfying \eqref{eq:enhance-1} or $\phi_1(E_1),\dots,\phi_K(E_K)$ satisfying 
\eqref{eq:enhance-1p},
  the FDR of the e-BH procedure is at most $  \alpha$. For $b_k\ge 1$, let 
\begin{equation}
\label{eq:zalphak}
z_{\alpha,k} (b_k) = \max_{x\in K/\mathcal K}  x  {\p(\alpha b_k E_k \ge x)}.
\end{equation}
In step 1 of the e-BH procedure, we need to choose a boosting factor $b_k\ge 1$ such that $z_{\alpha,k} (b_k)\le \alpha$ (ideally, an equality). 
Under an extra condition, 
a suitable choice of $b_k$ admits a simple formula.

\begin{proposition}\label{prop:bk}
Suppose that $E_k$ is a continuously distributed null e-value. 
Let  $q_{1-\alpha}(E_k)$ be the left $(1-\alpha)$-quantile of $E_k$.
If
 \begin{equation}
\label{eq:e-cond}
t\mapsto t \p(  E_k\ge t) \mbox{~is decreasing on $[q_{1-\alpha}(E_k),\infty)$},
\end{equation}  
then $b_k:=(\alpha q_{1-\alpha} (E_k))^{-1}$ is the largest boosting factor which  satisfies \eqref{eq:enhance-1}. 
\end{proposition}

Condition \eqref{eq:e-cond} is not uncommon as $ \alpha$ is typically small and the value $t\p( E_k\ge t) $ goes to $0$ as $t\to \infty$  since $E_k$ has a finite mean. Condition \eqref{eq:e-cond} is satisfied by, for instance, the e-values in Example~\ref{ex:calibrator} as well as their distributional mixtures.   
For a continuously distributed $E_k$, condition \eqref{eq:e-cond} can be equivalently expressed as 
 $$
\beta \mapsto \beta q_{1-\beta}(E_k)  \mbox{~is decreasing on $[1-\alpha,1)$}.
 $$
 
 Similar to  $\bar y_{\alpha,k}$ in \eqref{eq:yalpha2},
there is a conservative version of  $z_{\alpha,k}$ which does not depend on $K$, given by
\begin{equation}\label{eq:zalpha2}
\bar z_{\alpha,k} (b_k) = \sup_{x\ge 1}  x  {\p(\alpha b_k E_k \ge x)}.
 \end{equation} 
 In particular, if \eqref{eq:e-cond} holds, then 
 $\bar z_{\alpha,k} (b_k)=  z_{\alpha,k} (b_k)$ 
 for the choice $b_k$ in Proposition~\ref{prop:bk}.

\begin{example}\label{ex:calibratorPRDS} We look at the e-values in Example~\ref{ex:calibrator}.
For $k\in \mathcal N$,  we have
$$t\p(  E_k\ge t)  =t(\lambda   /t)^{1/(1-\lambda)} = (\lambda   t^{-\lambda})^{1/(1-\lambda)},$$
which is a decreasing function in $t$, and thus \eqref{eq:e-cond} holds.
In this case, for $b\ge 1$,  
$$z_{\alpha,k} (b) = \p( \alpha b E_k \ge 1)  = (\lambda  \alpha b )^{1/(1-\lambda)},
$$
and by Proposition~\ref{prop:bk}, the best choice of $b_k$ is 
$$
b  =(\alpha q_{1-\alpha}(E_k))^{-1} =(\lambda \alpha^{\lambda})^{-1}.
$$ 
Since $\bar y_{\alpha,k} (b)  =  (  \lambda^\lambda \alpha b)^{1/(1-\lambda)}$ in Example~\ref{ex:calibrator},
 we get 
$z_{\alpha,k} (b) /\bar y_{\alpha,k} (b)=\lambda$.
 Hence, the 
FDR is improved by a factor of roughly $\lambda$ from Theorem~\ref{th:e2} (arbitrary dependence) to Theorem~\ref{th:e1} (PRDS), which could be substantial if $\lambda$ is small.
For instance, if $\lambda=1/2$, then  $b =2 \alpha^{-1/2}$.  
For $\alpha=0.05$, we have $b \approx 8.94$, which should be compared with $b \approx 6.32$ in Example~\ref{ex:calibrator} under arbitrary dependence.
\end{example}

 \begin{example}\label{ex:LR2}
For the setup of e-values  in Example~\ref{ex:LR},  \eqref{eq:e-cond} does not always hold. .
Nevertheless, for $b\ge 1$, $ z_{\alpha,k} (b ) $ and $\bar z_{\alpha,k} (b ) $ in \eqref{eq:zalphak} and \eqref{eq:zalpha2} have simple formulas  
 $$
z_{\alpha,k} (b ) = \max_{x\in K/\mathcal K}  x\,\Phi \left( \frac{\log (\alpha b ) }{\delta} -\frac{\log x }{\delta}- \frac{\delta}{2} \right),
 $$
and
 $$
\bar z_{\alpha,k} (b ) = \max_{x\ge 1}  x\,\Phi \left( \frac{\log (\alpha b ) }{\delta} -\frac{\log x }{\delta}- \frac{\delta}{2} \right) .
 $$ 
Without specifying $K$, we use $\bar z_{\alpha,k} (b )$, and the value $b$ for $\bar z_{\alpha,k} (b ) =\alpha$ can be easily computed numerically. 
For instance, if $\delta =3$ and $\alpha=0.05$, then  $b\approx 7.88$ (compared with $b \approx 1.37 $  in Example~\ref{ex:LR}), and if $\delta =4$ and $\alpha=0.05$, then  $b\approx 10.31$ (compared with $b \approx 1.11 $  in Example~\ref{ex:LR}).
\end{example}

\subsection{An optimality result}
\label{sec:7}

There is a simple way of generating 
other e-testing procedures similar to the base e-BH procedure by transforming the   e-values.
In this section, working under the assumption that 
no distributional information on e-values is available, 
we focus on a procedure resulting from applying a common transform to all raw e-values. 
We obtain an optimality of the base e-BH procedure among all such transforms. 
 
Take a strictly increasing and continuous function $\phi:[0,\infty]\to [0,\infty]$ with $\phi(\infty)=\infty$ and $\phi(0)<1$. 
We shall call $\phi$ an \emph{increasing transform}.  
We design an e-testing procedure $\cG(\phi)$ by rejecting the $k_{e,\phi}^*$ hypotheses with  the largest e-values,
where 
\begin{equation}
\label{eq:e-k2}
k_{e,\phi}^*=\max\left\{k\in \mathcal K: \frac{k \phi(  e_{[k]})}{K} \ge1\right\}. \end{equation}
 It is clear that the choice of $\phi: t\mapsto \alpha t$ corresponds to $\tau_\phi=t_\alpha$ in Section~\ref{sec:4}, which yields the base  e-BH procedure $\cG(\alpha)$. 

Using  Theorem~\ref{th:e2} by choosing the boosted e-values as $\phi(E_k)/\alpha$, $k\in \mathcal K$, 
the FDR of $\cG(\phi)$  satisfies   \begin{equation}
 \label{eq:FDR-gen-e} \E\left[\frac{F_{\cG (\phi)}}{R_{\cG(\phi)} }\right] 
\le \frac{K_0}{ K}y_{\phi},  
\end{equation}   where 
$$
y_{\phi}:=\frac 1 {K_0} \sum_{k\in \mathcal N} \E[T(\phi(E_k))].
$$

The special choice $\phi:t\mapsto \alpha t$, corresponding to the  e-BH procedure, always guarantees 
$y_\phi=  y_\alpha \le \alpha$, where $y_\alpha$ is defined in \eqref{eq:yalpha}.
For a general choice of $\phi$, 
the FDR of $\cG(\phi)$ requires the knowledge of $y_{\phi}$, or at least an upper bound.

In the next result, we show that, if we require an FDR guarantee for arbitrary e-values, then the base e-BH procedure is optimal among the class $\cG(\phi)$. 
This is reassuring for us to use the base e-BH procedure as a canonical candidate for handling arbitrarily dependent e-values without further information. 

\begin{theorem}
\label{th:optimal}
Fix $\alpha \in (0,1)$ and $K$. For any increasing transform $\phi$, if $\cG(\phi)$   satisfies 
  $$
\E\left[\frac{F_{\cG(\phi)}}{R_{\cG(\phi)} }\right] \le  \alpha 
$$  for arbitrary configurations of e-values, then $\cG(\phi)\subseteq \cG(\alpha)$.
\end{theorem}

\subsection{Applying the e-BH procedure to p-values}
\label{sec:gen-p}

  In this section, we compare the BY procedure  
\[
\cD'(\alpha) := \cD(\alpha/\ell_K)
\]
and the e-BH procedure $\cG(\alpha)$, as they both require no assumptions on the dependence structure. Although the two procedures are comparable on the dependence assumption, we remark that e-values generally use less information than p-values, as  p-values require  full specification of the distributions of the test statistics, whereas e-values only require the information of a known mean. Note that information on the boosting factors $b_k$, $k\in \mathcal K$  is contained in the null distribution of the test statistics.

To properly compare the p- and e-based procedures, we need to calibrate between p-values and e-values using calibrators. 
With such a calibration, the e-BH procedure gives rise to a general class of p-testing procedures with FDR guarantee for arbitrary p-values, similar to $\cG(\phi)$ in Section~\ref{sec:7}, 
and this class includes the base BH and the BY procedures as  special cases.
On a related note,
calibration to e-values serves as a crucial intermediate step in combining p-values under arbitrary dependence for testing a global null \citep[Section 5]{VWW20}.

Let $\psi:[0,1]\to [0,\infty]$ be a strictly decreasing and continuous function with $\psi(0)=\infty$. 
We shall call $\psi$ a \emph{decreasing transform}.  
We design a p-testing procedure $\cD(\psi)$  by rejecting    the $k_\psi^*$ hypotheses with the smallest p-values, where
 $$
k_\psi^*=\max\left\{k\in \mathcal K:  \psi(p_{(k)}) \ge \frac{K}{k }\right\},$$
with the convention $\max(\varnothing)=0.$  
A smaller transform function leads to less power of the testing procedure.  
A similar procedure using   different transforms  on individual p-values is analyzed in Appendix~\ref{app:multi}.
An important example of $\cD(\psi)$ is to choose  
$\psi:p\to \alpha/p$ for some $\alpha \in(0,1)$.  Note that   $$  \psi(p_{(k)})   \ge \frac{K}{k   } ~~~~\Longleftrightarrow ~~~~
\frac{K  p_{(k)} }{k }  \le   \alpha  .
 $$
In this case, $\cD(\psi)$ is precisely the BH procedure $\cD(\alpha)$.
 
 The  p-testing procedure $\cD(\psi)$ is equal to 
the step-up procedure of \cite{BY01} which rejects $k^*:=\max\{k\in \mathcal K: p_{(k)} \le \alpha_k\}$ hypotheses with the smallest p-values,
where $\alpha_k=\psi^{-1}(K/k)$, $k\in \mathcal K$.
Such a procedure also appears in  \cite{BR08}  where  $\alpha_k= \alpha \beta(k)/K$ and  $\beta$ is called a shape or reshaping function.
Our main purpose here is to compare the p- and e-testing procedures, not to propose new step-up p-testing procedures.

The objects $\phi(p_1),\dots,\phi(p_K)$ can be treated as boosted e-values. Hence, 
using Theorems~\ref{th:e2} and~\ref{th:e1},
we can easily calculate the FDR of $\cD(\psi) $.

\begin{proposition}
\label{prop:psi}
  For arbitrary p-values and a decreasing transform $\psi$, the testing procedure $\cD(\psi)$ satisfies 
$$
\E\left[\frac{F_{\cD(\psi)}}{R_{\cD(\psi)} }\right] \le \frac{    K_0}{K}  y_\psi,
$$
where 
\begin{equation}\label{eq:ypsi}
y_\psi:= \psi^{-1}(1) + \sum_{j=1}^{K-1} \frac{K}{j(j+1)}
 \psi^{-1} ( K/j )  .\end{equation}
If the p-values are PRDS, then  
$$
\E\left[\frac{F_{\cD(\psi)}}{R_{\cD(\psi)} }\right] \le \frac{    K_0}{K} z_\psi ,
$$ 
where 
\begin{equation}\label{eq:zpsi}
z_\psi: =\max_{t \in K/\mathcal K} t\psi^{-1}(t).
\end{equation}

\end{proposition}

If the decreasing transform  $\psi$ satisfies 
\begin{equation}
\label{eq:e-cond2}
t\mapsto t \psi^{-1}(t) \mbox{~is decreasing on $[1,\infty)$},
\end{equation} similarly to condition \eqref{eq:e-cond}, 
 we have $z_\psi=\psi^{-1}(1)$. 
We can replace $[1,\infty)$ by $[1,K]$ or $K/\mathcal K$ in \eqref{eq:e-cond2}; the current condition \eqref{eq:e-cond2} is slightly stronger but it does not depend on $K$.

For the specific choice $\psi:p\to \alpha/p$ which gives the BH procedure, 
we have $\psi^{-1}=\psi$,   $$y_\psi=  \alpha  + \sum_{j=1}^{K-1} \frac{\alpha  }{j+1}
 =\alpha \ell_K, \text{ and } z_\psi= \max_{t \in K/\mathcal K} t\psi^{-1}(t) = \alpha.$$
 By Proposition~\ref{prop:psi},  $\cD(\psi) =\cD(\alpha)$ has 
 FDR guarantee $K_0 \ell_K \alpha /K$ for arbitrarily dependent p-values,
 and FDR guarantee $K_0\alpha/K$ for PRDS p-values. 
  This gives an analytical  proof  of  the FDR guarantee of the BH and BY procedures in Theorem~\ref{th:BH}.

Let $f$ be a calibrator and $\psi=\alpha f$. We can see that $y_\psi \le \alpha$ by Theorem \ref{th:e2} since $f(p_1),\dots,f(p_K)$ are e-values. Conversely, for any decreasing transform $\psi$ satisfying $y_\psi\le \alpha$ and  taking values in $K/\mathcal K$ (recall that only values of $\psi$ in $K/\mathcal K$ matter), 
the function $f=\psi/\alpha$ 
is a calibrator
since $y_\psi=\int_0^1 \psi (p)\d p= \alpha \int_0^1 f(p) \d p$. Therefore, all $D(\psi)$ with $y_\psi\le 
\alpha$ can be obtained via calibration to e-values. 

\begin{remark}
As we mentioned above, via the relationship 
\begin{align}\label{eq:reshape2}\psi^{-1}(K/k) = 
   \alpha \beta(k)/K, \end{align} 
   one obtains the step-up procedure based on the shape function $\beta$. Hence, such procedures can also be equivalently expressed via calibration to e-values.
   In \cite{BR08}, an important condition on $\beta$ is
   \begin{align}\label{eq:reshape} 
   \beta(k)=\int_0^k x \d \nu(x),~k\in \mathcal K, \mbox{ for some probability measure $\nu$ on $(0,\infty)$}.
   \end{align}
   Indeed, if  \eqref{eq:reshape2} and \eqref{eq:reshape} hold, then 
   by \eqref{eq:ypsi},
  $$ y_\psi =\alpha \sum_{j=1}^K \frac{\beta(j)-\beta(j-1)}{j} = \alpha \sum_{j=1}^K \int_{j-1}^j  \frac{x}j \d \nu (x) \le \alpha \nu((0,\infty))=\alpha.$$
  Therefore, Proposition \ref{prop:psi} gives that $D(\psi)$ has FDR level at most $\alpha$, recovering the result in \cite{BR08}.
  Moreover, we can check that \eqref{eq:reshape} is also sufficient for \eqref{eq:e-cond2}, and thus $D(\psi)$ has FDR level at most $\alpha \beta(K)/K$ under PRDS by Proposition \ref{prop:psi}.
\end{remark}

   The inequality  
$\alpha \E\left[ t_\alpha f(t_\alpha) \right] \le  K_0z_\alpha$ in the proof of Theorem~\ref{th:e1}  is generally not  an equality.
Hence, the FDR bound provided by Proposition~\ref{prop:psi} may not be sharp in the case of PRDS p-values. 
Nevertheless,
in the next proposition we shall see that this bound is almost sharp under some extra conditions. From there, we obtain a weak optimality result on the base BH procedure.
\begin{proposition}
\label{prop:psi3}
Fix $\alpha \in (0,1)$ and $K$. For any decreasing transform $\psi$,  
if $\cD(\psi)$ satisfies 
$$
\E\left[\frac{F_{\cD(\psi)}}{R_{\cD(\psi)} }\right] \le    \alpha
$$ 
for arbitrary configurations of  PRDS p-values, 
then $\psi^{-1}(1)\le \alpha$. 
Moreover, if $\psi$ satisfies \eqref{eq:e-cond2}, then $\cD(\psi)\subseteq \cD(\alpha)$.  
\end{proposition}

 Comparing the optimality of the e-BH procedure in Theorem~\ref{th:optimal}  with the optimality of the BH procedure  in Proposition~\ref{prop:psi3}, we note two differences: 
Theorem~\ref{th:optimal} is stated for arbitrary e-values whereas  Proposition~\ref{prop:psi3} is stated for PRDS p-values;
 Theorem~\ref{th:optimal}  imposes no assumption on $\phi$ whereas Proposition~\ref{prop:psi3} requires $\psi$ to satisfy \eqref{eq:e-cond2}.

Different from the case of PRDS p-values, 
 $\cD'(\alpha)$  
 and $\cD(\psi)$ may not be strictly comparable 
for arbitrarily dependent p-values.
We discuss this issue in Example~\ref{ex:compare} below.
The general message is that, in contrast to the case of PRDS p-values as shown in Proposition~\ref{prop:psi3}, the BY procedure is not necessarily always the best for arbitrarily dependent p-values.

\begin{example}\label{ex:compare}
We consider the calibrators in \eqref{eq:ex-2} 
by choosing $\psi : p \mapsto  \theta \lambda p ^{\lambda-1}$ for some $\theta >0$.
For simplicity, we take $\lambda =1/2$ (a similar procedure was proposed by \cite{Sarkar08}).
As we see from Example~\ref{ex:calibrator}, 
  $\cD(\psi)$ has an FDR guarantee of $y_\psi\le  \theta ^2/2$.
  To compare with $\cD'(\alpha)$, we choose $\theta = (2\alpha)^{1/2}$,
  so that both procedures have FDR guarantee of $\alpha$.
Let $k^*$ be the number of  hypotheses rejected by $\cD'(\alpha  )$.
Note that a sufficient condition for $k_\psi^*\ge k^*$ 
is 
$\psi(p_{k^*}) \ge K/k^*$. 
If we set
$$\gamma: = \frac{k^* \alpha  }{  K \ell_K} \approx p_{k^*}, $$
 then, approximately, the above condition is 
 $$
\psi(\gamma) \ge \frac K {k^*}~~~~\Longleftrightarrow~~~~\frac{k^*}{K} \ge \frac{2 }{ \ell_K}.
$$  
Note that $\ell_K\sim \log K$, and $k^*/K$ is the proportion of rejection among all hypotheses.   Hence,
$\cD(\psi)$ is more powerful than $\cD'(\alpha)$
roughly when the proportion of rejections exceeds $2/\log K$.
This conclusion is  independent of $\alpha$.
\end{example}

\section{Simulations: multiple testing for ordered hypotheses using multi-armed bandits}

\label{sec:simulation}

Here, 
we conduct simulation studies for a setting that merges multiple testing with an ordered variant of the multi-armed bandit problem.  Appendix~\ref{app:num} contains additional simulation results for a more classical non-sequential multiple testing setup involving z-tests. While it is more traditional to present the latter application in a paper, we find the former a more conceptually interesting case study.

 Consider a multi-armed bandit   with $K$ arms, where pulling arm $k$ produces an iid sample $(X^k_1,X^k_2,\dots)$ from a non-negative random reward $X^k$.
 The $K$ null hypotheses are $\E[X^k]\le 1$, $k=1,\dots,K$, which are nonparametric  and similar to Example~\ref{ex:21}. For interpretation's sake, think of $X^k_j$ as follows: each arm is an investment and $X^k_j$ is the return of the $k$-th investment on its $j$-th trial. Specifically, imagine that each investment $k$ starts with unit wealth and   the wealth  can be multiplied by $X^k_j$ when the arm is pulled the $j$-th time. 
 Our high-level aim is thus to find profitable investments as quickly as possible; although we do not write down a formal objective, this is not necessary to convey our main points. (One may equally think of testing if the $k$-th drug has an effect or not.) 
 
Consider a simple setting in which we have used side information or prior knowledge to arrange the arms in a prespecified order and we must collect data on arm $k$ before visiting arm $k+1$; thus previous arms cannot be revisited to obtain more samples. Each arm can be pulled at most $n$ times (the budget) before moving on to the next one. 

We describe a natural algorithm to tackle this problem (an algorithm consists of two parts: deciding when to stop pulling an arm and move to the next one, and deciding which nulls to reject at the end).
Our algorithm will use a multiple testing procedure $\cD$ such as BH or e-BH for both steps: we move onto the next arm if we have exhausted the budget for the current arm, or our multiple testing procedure can reject the null with the available information. To formalize this,
 we take an e-testing procedure as an example. (For a p-testing procedure,   e-values are replaced by p-values.) 
For $k=1,\dots,K$:
\begin{enumerate}
\item At the end of dealing with arm $k-1$, we can summarize the available information as a vector of e-values $\mathbf e_{k-1}:= (e_1,\dots,e_{k-1},1,\dots,1)$.
\item  Denote by $e_{k,j}$ the e-value obtained after pulling the $k$-th arm $j$ times. We pull arm $k$  and stop pulling after $T_k$ trials if either
\begin{enumerate}
\item $|\cD  (e_1,\dots,e_{k-1},e_{k,T_k},1,\dots,1)|>|\cD(\mathbf e_{k-1})|$, i.e., there is at least one new discovery, or
\item arm $k$ has been pulled $T_k=n$ times, i.e., the budget is exhausted.
\end{enumerate}  
\item Set $e_k:=e_{k,T_k}$ and move on to arm $k+1$.
\end{enumerate}
After all $K$ arms are finished, we apply the procedure $\cD$ to $\mathbf e_K:=(e_1,\dots,e_K)$ to determine the rejected hypotheses. 
 
Since we are in a nonparametric setting, the natural choices of e-values and p-values are based on e-processes.  More precisely, 
the e-value $e_{k,j}$ and the p-value $p_{k,j}$ are realized by, respectively, 
$$E_{k,j}:=\prod_{i=1}^j X_i^k \mbox{~~~and~~~} P_{k,j} :=\left(\max_{i=1,\dots,j} E_{k,i}\right)^{-1}, $$ 
for each $k=1,\dots,K$ and $j=1,\dots,n$. Also, define $E_{k,0}=P_{k,0}=1$ for simplicity. It is not hard to see that if $H_k$ is true then $(E_{k,j})_{j=0}^n$ is an e-process.
The final e-variables $E_{k}$ and p-variables $P_{k}$ are obtained by  $$
E_k= E_{k,T_k}~~~\mbox{and}~~~ P_k=P_{k,T_k},~~k=1,\dots,K.
$$ 
The fact that the above stopped e-processes yield valid p-values is obtained by invoking  Ville's inequality~\citep{SSVV11,HRMS20}. 

The above algorithm is a natural baseline in the sense that, for $\cD$ being BH, e-BH or BY, one cannot obtain more rejections than this algorithm, even if all arms are pulled $n$ times.
The algorithm may be seen as an ordered version of  that of \cite{JJ18}, but
 a detailed discussion is not necessary here.

It is important to note that, even assuming data across arms are mutually independent,  the produced e-values (or p-values) are dependent in a complicated way; a larger previous realization leads to a smaller threshold for the next stopping rule. Indeed, we have $T_k < n$ only if Step 2(a) was invoked, meaning that $e_{k,T_k} \geq \tfrac{K}{|\cD(\mathbf e_{k})| \alpha}$. Hence, some negative dependence exists and the BH procedure does not control FDR theoretically.
There are, however, two p-testing procedures that have valid FDR: the BY procedure $\cD(\alpha/\ell_K)$ which is designed for arbitrary dependence, and the \emph{self-consistent or compliant BH} (cBH) procedure  $\cD(\alpha')$,  where $\alpha'$ is given in \eqref{eq:alphaprime}; i.e., it satisfies $\alpha'(1-\log(\alpha'))=\alpha$   (e.g., $\alpha'=0.0087$ for $\alpha =0.05$).
Both the BY and the cBH procedure give a   FDR guarantee of $\alpha$ in our experiments, although  cBH relies on the independence of data across arms (BY and e-BH do not).\footnote{We briefly explain the validity of cBH  in our setting, which is inspired by observations of \cite{JJ18}. First,  define the \emph{unobserved} latent p-values $P'_1,\dots,P'_K$ by $P'_k= P_{k,n}$, $k\in \mathcal K$. Note that $P'_1,\dots,P'_K$ are independent --- since every arm is sampled $n$ times, there is no adaptivity, and thus no dependence.
Further, $P_k\ge P'_{k}$ by definition, and so
the output (rejection set) of applying BH to $P_1,\dots,P_K$ is a subset of that of applying BH to $P'_1,\dots,P'_K$.
Therefore, applying BH to $P_1,\dots,P_K$ can been as a (randomized) self-consistent procedure applied to the  independent latent p-values $P'_1,\dots,P'_K$, and its FDR control under the correction \eqref{eq:alphaprime} follows  by noting that the proof technique in \cite{Su18}  directly applies to randomly selected self-consistent rejection sets.}

Therefore, we will compare the four procedures  e-BH, BH, BY and cBH  in our numerical experiments, while keeping in mind that BH does not have a valid FDR in theory, although 
by design it has more rejections than the other three procedures. 
We summarize the allowed conditions for a valid FDR guarantee for these four methods in Table~\ref{tab:mbp-valid}. 

\begin{table}[th] 
  \begin{center} 
    \caption{Conditions for the validity of the testing algorithm. To allow for dependence among data across arms, we only require for $k\in \mathcal N$, $\E[X^k_j|\mathcal G]\le1$ where $\mathcal G$ is  the available information before we observe $X^k_j$ (weaker than independence used in our experiments); see Example~\ref{ex:21}.}
    \label{tab:mbp-valid}
  \small  
    \begin{tabular}{rccc }
            &  {arbitrarily dependent} &   {arbitrarily dependent}  & FDR guarantee in    \\ 
 & data   across arms & stopping rules $T_k$  &     {our experiments}
 \vspace{2pt}      \\  \hline
\rule{0pt}{3ex}   e-BH & YES  &  YES & valid  at level  $\alpha K_0/K$  \\
BH   & NO & NO   & not valid \\
BY   & YES  & YES    &  valid  at level  $\alpha K_0/K$     \\
cBH & NO &  YES  &   valid  at level  $\alpha K_0/K$     \\
    \end{tabular} 
  \end{center} 
 \end{table}

Below we describe the data generating process used in our experiments (note that the design and the validity of   the above procedures  do not depend on any knowledge of the data generating process). As mentioned above, arms are naturally ordered  such that more promising arms 
come first.
More specifically, arm $k$ is non-null with probability $\theta (K-k+1)/(K+1)$ where $\theta\in [0,1]$ is a parameter. The expected number of non-nulls in this setting is $\theta/2$. 
Further,   let $s_k$ be the strength of signal in each non-null hypothesis, 
which follows an iid exponential distribution with mean $\mu$. With this setting, some non-nulls may have a very weak signal and they are almost impossible to  distinguish  from a null, whereas some other non-nulls with a strong signal may be rejected   quickly with a few pulls. 
Conditional on $s_k$, the data $X^k_1,\dots,X^k_n$  for arm $k$ are iid following a log-normal distribution 
$$
X^k= \exp\left(Z^k+s_k\id_{\{k\in\mathcal K\setminus \mathcal N\}} -1/2\right) 
$$
where $Z^1,\dots,Z^K$ are iid standard normal.  
Note that the data  are independent across arms, and hence the cBH procedure is valid (see Table~\ref{tab:mbp-valid}).

In the experiments, we set $\alpha=0.05$ and $\theta=1/2$ (on average, $1/4$ of all hypotheses are non-null).   
Results for some values of other parameters are summarized in Table~\ref{tab:mbp-2}.

 \begin{table}[th] 
\centering
  \begin{center} 
  \caption{Simulation results  in the multi-armed bandit setting,
  where 
    $R$ is the number of rejected hypothesis, 
$B$ is the proportion of unused budget (number of unused pulls divided by $nK$), 
and  TD is the number of true discoveries (rejected non-null hypotheses).
  Each number is computed over an average of 500 trials. 
  The default values of parameters are  $K=500$, $n=50$ and $\mu=1$,
  and each panel may have one parameter value different from the default.
  }
    \label{tab:mbp-2} 
       \begin{subtable}[t]{0.49\textwidth} 
    \caption{Default 
    }\centering
     \begin{tabular}{rcccc}
            & $R$ &$B\%$ & TD    & FDP\% \\ \hline
e-BH &74.4 &11.42 &73.2 &1.58\\
BH   &77.0 &11.44 &75.3 &2.13\\
BY   &70.6 &10.06 &70.4 &0.31\\
cBH  &71.1& 10.16 &70.8& 0.36\\
    \end{tabular}
      \end{subtable}   
        \begin{subtable}[t]{0.49\textwidth} 
    \caption{$K=2000$ 
    }\centering
    \begin{tabular}{rcccc}
            & $R$ &$B\%$ & TD    & FDP\% \\ \hline 
e-BH & 297.6& 11.39 &293.2 & 1.48\\
BH   &307.8 &11.41 &301.4 &2.07\\
BY  & 281.2 & 9.95 &280.4& 0.26\\
cBH  & 284.5 &10.15& 283.5& 0.36\\
    \end{tabular}
    \end{subtable} 
            \\
    ~
    \\
      \begin{subtable}[t]{0.48\textwidth} 
    \caption{$n=10$ 
    }\centering
    \begin{tabular}{rcccc}
            & $R$ &$B\%$ & TD    & FDP\% \\ \hline 
e-BH &47.7& 3.99& 47.3 & 0.83\\
BH   &49.3 &4.01& 48.7 & 1.06\\
BY   &38.4& 2.77 &38.4& 0.08\\
cBH  &39.2& 2.85 &39.2 &0.11\\
    \end{tabular}
      \end{subtable}  
      ~  
        \begin{subtable}[t]{0.48\textwidth} 
    \caption{$n=100$ 
    }\centering
    \begin{tabular}{rcccc}
            & $R$ &$B\%$ & TD    & FDP\% \\ \hline 
e-BH &79.1 &13.48 &77.9 &1.50\\
BH  & 81.3 &13.50 &79.5 &2.13\\
BY   &76.4 &12.36 &76.1 &0.35\\
cBH  &76.7& 12.44 &76.4& 0.41\\
    \end{tabular}
    \end{subtable} 
            \\
    ~
    \\
      \begin{subtable}[t]{0.48\textwidth} 
    \caption{$\mu=0.5$ 
    }\centering
    \begin{tabular}{rcccc}
            & $R$ &$B\%$ & TD    & FDP\% \\ \hline
e-BH & 43.5 &5.77 &42.9 &1.54\\
BH & 46.3& 5.80 &45.3& 2.13\\
BY & 39.6 &4.66 &39.5& 0.27\\
cBH &40.1 &4.74 &40.0& 0.35\\
    \end{tabular}
      \end{subtable}  
      ~ 
        \begin{subtable}[t]{0.48\textwidth} 
    \caption{$\mu=2$ 
    }\centering
    \begin{tabular}{rcccc}
            & $R$ &$B\%$ & TD    & FDP\% \\ \hline 
e-BH &97.4& 16.46 &95.9& 1.54\\
BH  & 99.3 &16.47& 97.2 &2.07\\
BY  & 94.3& 15.23 &94.1 &0.29\\
cBH & 94.6& 15.32 &94.3& 0.35\\
    \end{tabular}
    \end{subtable} 

  \end{center} 
  
 \end{table}

In all settings of Table~\ref{tab:mbp-2}, 
 we   see that the e-BH procedure outperforms both BY and cBH by providing more discoveries and using a smaller number of pulls. 
 The BY and cBH procedures are  both quite conservative and they perform similarly, with cBH being  better for large $K$ as expected. 
The numbers of true discoveries of e-BH  are  slightly lower than those of  BH, but, as we explained above,   BH   does not have a theoretical guarantee of FDR   (its realized FDP is not larger than $\alpha$ in our experiments, but this may be due to the fact that the null p-values are not uniform).
 
 For many other settings of parameters and distributions of $s_k$ that we tried,
the relative performance of the four methods is always qualitatively similar.

\section{A real-data example: detecting cryptocurrencies with positive expected return}
\label{sec:realdata}
 
 In this section we provide a real-data example in the spirit of Example \ref{ex:21} to detect cryptocurrencies (which we call \emph{coins}, like \texttt{Bitcoin}, for brevity) with positive expected return.  This analysis can be easily conducted for other financial assets or trading strategies.

We collect daily price data of coins from Jan 4, 2015 to Jun 30, 2021 (price data are obtained from \texttt{CoinGecko}). We choose to start from 2015 so that we have a few hundreds of coins to work with. 
There were 496 coins that had trading prices on Jan 4, 2015 (the list is from \texttt{CoinMarketCap}), among which 126 were active (``alive") on Jun 30, 2021.
We order the coins by their market capitalization in Jan 2015 from the largest to the smallest (e.g.,  the first being \texttt{Bitcoin}).

Let $Y_{k,j}$ be the price of  coin  $k$ at the end of the $j$-th month, $j=1,\dots,T$ and $k=1,\dots,K$. In our data set, $T=78$ and $K=496$ (or $100$, $200$; see later). 
Let
 $X_{k,j}=Y_{k,j}/Y_{k,j-1}$ be the $j$-th month growth of the coin $k$, $j\ge 1$, where $Y_{k,0}$ is the initial price of coin $k$.
The  $k$-th  null hypothesis  is  $\E[X_{k,j} \mid \mathcal F_{j-1}] \le 1$ for $j=1,\dots,T$; that is, during the considered period of time, coin $k$ has a monthly expected growth no larger than $1$ given previous market information $\mathcal F_{j-1}$; see Example \ref{ex:21}.
We are interested in which coins  generate positive expected return during the period. (However, our goal is not to predict future price movements, on which we do not have much to say.)  
 
For each coin $k$, we start with an initial wealth of $\$1$,  invest $\lambda \in [0,1]$ proportion of our money into the coin and rebalance every month. We ignore transaction costs. The portfolio wealth process $ (W_{k,t})_{t=0,1,\dots,T}  $ is given by
$$W_{k,0}=1 ~~~\mbox{and}~~~
W_{k,t} = \prod_{j=1}^t (1-\lambda + \lambda X_{k,j}),~~~t=1,\dots,T.
$$
Under the $k$-th null hypothesis, $(W_{k,t})_{t=0,1,\dots,T} $ is an e-process.  
Similarly to Section \ref{sec:simulation}, 
we construct e-values and p-values by $$E_{k}:= W_{k,T} \mbox{~~~and~~~} P_{k} :=\left(\max_{t=0,1,\dots,T} W_{k,t}\right)^{-1}.$$ In general, $\lambda$ is allowed to depend on time and past data, and that would correspond to an adaptive trading strategy; we choose $\lambda$ to be a constant here for simplicity. Note that $E_k$ above does not require a disclosure of the intermediate wealth values, which may not be available in more complicated applications where privacy or propriety is concerned. 

We consider two simple strategies, $\lambda=1$ which corresponds to buying and holding the coin, and $\lambda =1/2$ which corresponds to monthly rebalancing to maintain half wealth invested in the coin.  
Note that for $E_k$ with $\lambda =1$, we do not need intermediate price data other than the initial and the terminal prices. The vanished coins (dead during the considered period, price data not included in \texttt{CoinGecko})  are treated as having insignificant p-values and e-values; recall that for the procedures in this paper, p-values larger than $\alpha$ and e-values less than $1/\alpha$ can be safely treated as $1$. 

Among the 496 coins listed in Jan 2015, many had tiny market capitalization and did not have  meaningful trading activities. For this reason, investing in the full set of 496 coins may not be the most appropriate due to liquidity. 
Instead,  it may be sensible to invest only in 200 or 100 coins with the largest market capitalization in Jan 2015.
We thus consider the three cases of $K=496$ (126 alive in Jun 2021), $K=200$ (92 alive) and $K=100$ (54 alive).

The coin returns and the portfolio values are highly dependent in a complicated way. We will apply   the e-BH procedure and the BY procedure to the e-values and p-values to select coins, as both methods are valid under arbitrary dependence. 
 The results are reported in Table \ref{tab:r1-2}.
 
%
%
%
    \begin{table}[th] 
  \begin{center} 
    \caption{Number of coins selected by e-BH and BY  procedures}
    \label{tab:r1-2} 
    \begin{tabular}{rccccccccc}
            &  \multicolumn{3}{c}{buy and hold}  &   \multicolumn{3}{c}{50\%-rebalancing}    &   \multicolumn{3}{c}{30\%-rebalancing}   \vspace{2pt}    \\ 
\#total coins &  ~496~ & ~200~ & ~100~  & ~496~ &  ~200~     & ~100~ & ~496~ &  ~200~      & ~100~
 \vspace{2pt}      \\  \hline
\rule{0pt}{3ex}  e-BH  5\%  &3  & 4   & 0 &42  &28  & 13  &21  &12  & 6 \\ 
    BY   5\% &11  &0 & 0  &  15 &11&  6 &  15 &4&  5  \\ 
 e-BH    10\%  &13 & 7      & 2 & 60 &48& 24 & 32 &22& 8      \\  
 BY 10\%    & 18 & 7      & 0  & 25 &15 & 7 & 13 &7 & 7  \\ 
    \end{tabular} 
  \end{center} 
 \end{table} 
  
We observe that 
the e-BH procedure produces more discoveries than the BY procedure in almost all considered settings, except for $(K,\lambda)=(496,1)$.
In addition, we impose essentially no model assumptions in this application, and the  e-BH procedure can be applied even if only initial and final prices are observed.  On a side note,  for most coins, the 50\%-rebalancing strategy produces higher terminal wealth $W_{k,T}$ than the buy-and-hold  and 30\%-rebalancing strategies,  and this could  partially be explained by a combination of the very high volatility  and the high return of coin prices during the considered period.


\section{Conclusion}
 
We have introduced the e-BH procedure to achieve FDR control with e-values. Some highlights of the e-BH procedure are summarized below:
\begin{enumerate}
\item It works for arbitrarily dependent e-values (Theorem \ref{th:validity});
\item it requires no information on the configuration of the input e-values, and also works for weighted e-values (Section \ref{sec:52});
\item it allows for boosting  e-values if partial distributional information is available on some e-values (Section \ref{sec:42});
\item it gives rise to a class of p-testing procedures which include both the BH procedure and the BY procedure as special cases along with proofs of FDR control (Proposition~\ref{prop:psi});
\item it is optimal among the class of e-testing procedures $\cG(\phi)$ in the setting of arbitrary e-values  (Theorem~\ref{th:optimal}). 
\end{enumerate}

Although the e-BH procedure is primarily designed for cases where e-values are directly available (not calibrated from bonafide p-values), even the p-value calibrated e-BH procedure may outperform the BY procedure in some settings (see the simulation results in Appendix~\ref{app:num}). 
There are two more complicated situations where the e-BH procedure can be applied:
\begin{enumerate}
    \item  Natural e-values are   available for some hypotheses, whereas natural p-values are   available for the others. Depending on proportion, one may calibrate them all to p-values or all to e-values. 
    \item A hypothesis has both an e-value and a p-value (or even multiple p-values and e-values), which may be dependent and obtained from different experiments. In such a situation, the power comparison of e- and p-based procedures depends on the quality of those e-values and p-values. One may   choose to combine them into a single e-value or p-value for each hypothesis, e.g., using the methods in \cite{VW20}. 
\end{enumerate}
Thorough theoretical and empirical studies of the last two ``hybrid'' situations are needed to better understand the comparative advantages of p-values and e-values.

Finally,  an important issue on the FDR literature is to incorporate certain  assumptions on the dependence structure among p-values. A popular choice is to impose a multivariate Gaussian structure; see e.g., \cite{Guo2014},   \cite{Barber2015},  \cite{DR15} and \cite{FL20}. One may naturally wonder whether there is additional boosting for e-values in a multivariate Gaussian setting, and this leads to interesting questions for future studies.

 \subsubsection*{Acknowledgements}
 
 We thank an Editor, an Associate Editor, and two referees for helpful comments on an earlier version of the paper. We also thank Peng Ding, Lihua Lei, Song Mei, and Vladimir Vovk   for many useful discussions, and  Ziyu Chi for excellent research assistance. 
 AR acknowledges NSF DMS CAREER grant 1945266.
 RW acknowledges financial support from NSERC grants  (RGPIN-2018-03823, RGPAS-2018-522590).  
 
\bibliographystyle{apalike}
\bibliography{local}

 \appendix 
 
 \section{Simulations for correlated z-tests}\label{app:num}
  
We provide further simulation results in a classic setting, where test statistics $X_k,~k\in\mathcal K$  are generated  from correlated z-tests.  
The null hypotheses are $ \mathrm{N}(0,1)$ and the alternatives are $ \mathrm{N}(\delta,1)$,
where we take $\delta:=-3$ throughout the section.
We generate $ K-K_0$ observations from  the alternative distribution $ \mathrm{N}(\delta,1)$
and then $K_0$ observations from   the null distribution $\mathrm{N}(0,1)$.
  
We take the p-values as
\begin{equation}\label{eq:base-p}
  P(x)
  :=
  \Phi(x),
\end{equation} 
where $x  \in \{X_k: k\in \mathcal K\}$;
these are the p-values found using the most powerful test given by the Neyman--Pearson lemma. 
The raw e-values are, following \cite{VW20}, the likelihood ratios
\begin{equation}\label{eq:base-e}
  E(x)
  :=
  \frac{\exp(-(x-\delta)^2/2)}{\exp(-x^2/2)}
  =
  \exp(\delta x - \delta^2/2)
\end{equation}
of the alternative to the null density.

We report the numbers of discoveries for  the  (base) BH, BY, e-BH (boosted), and base e-BH procedures, in several different settings in Table~\ref{tab:ztest1}. 
In particular, we consider two settings of negative correlation (thus, PRDS fails to hold). 
We set the target FDR level $\alpha\in \{10\%,5\%,2\%\}$.
All numbers in Table~\ref{tab:ztest1} are produced an the average from 1,000 trials. 

From  the simulation results in Table~\ref{tab:ztest1}, we can see that the BH procedure has slightly more discoveries than the e-BH procedure for PRDS e-values via \eqref{eq:enhance-1}; recall that we do not expect   e-BH   to outperform   BH    when PRDS and precise (standard uniform) p-values are available (cf.~Proposition~\ref{prop:psi3}).
The e-BH procedure  for arbitrarily dependent (AD) e-values 
performs slightly better than the BY procedure in most settings, 
and this advantage is more pronounced in the case of a large experiment (Table~\ref{tab:sim-3}, left), as the BY correction $\ell_K$ is penalized by a large $K$.
Only in the case of sparse signal and small $\alpha$ (Table~\ref{tab:sim-3}, last two columns), BY performs better than e-BH.
This is consistent with Example~\ref{ex:compare} (with different e-values from our experiments), where we see 
that BY outperforms e-BH if the proportion of rejections is very small.

These simulation results suggest that, even in case p-values are available, the e-BH procedure (with boosted e-values) has competitive performance in most situations, especially when PRDS fails to hold (thus the base BH does not have a correct FDR in theory).

\begin{table}[th] 
  \begin{center} 
    \small \caption{Simulation results for correlated z-tests, where $\rho_{i,j}$ is the correlation between two test statistics $X_i$ and $X_j$ for $i\ne j$.   
Each cell gives the number of rejections and, in parentheses, the realized FDP (in $\%$). 
   Each number is computed over an average of 1,000 trials.  
  }
  \setlength\tabcolsep{4pt} 
    \label{tab:ztest1}  
             \begin{subtable}[t]{0.99\textwidth} 
 \begin{center} \caption{Independent and positively correlated  tests, $K=1000$, $K_0=800$}
  \label{tab:sim-1} 
    \centerline{   \begin{tabular}{l|rrr|rrr}
    &\multicolumn{3}{c|}{$\rho_{ij}=0 $} &\multicolumn{3}{c}{$\rho_{ij}=0.5$}\\
       \rule{0pt}{3ex}     & \multicolumn{1}{c}{$\alpha= 10\%$} &  \multicolumn{1}{c}{$\alpha=5\%$} & 
        \multicolumn{1}{c|}{$\alpha=2\%$} &  \multicolumn{1}{c}{$\alpha=10\%$} &  \multicolumn{1}{c}{$\alpha=5\%$} &  \multicolumn{1}{c}{$\alpha=2\%$} \\
      \hline
 BH    &  177.3 (8.01)& 148.7 (4.07) &115.0 (1.63)  & 180.0 (7.00)& 144.8 (3.64) &109.8 (1.50)\\ 
     e-BH PRDS &171.8 (7.07) &147.6 (3.95) & 114.6 (1.62) &170.2 (5.71) &142.5 (3.35) &108.0 (1.50)\\\hline
BY        &101.1 (1.10)  &78.8 (0.57) & 53.2 (0.22)  & 96.6 (1.03) &76.7  (0.50)  &55.0 (0.20)\\
e-BH AD &  109.4 (1.41) & 85.4 (0.68)  &54.6 (0.24) &103.1 (1.32)  &81.4 (0.70)  &56.6 (0.28)\\\hline
base e-BH      & 97.5 (1.00) &70.6 (0.43)  &36.9 (0.11)  &91.9 (0.97)  &69.1 (0.45)  &43.6 (0.16)\\ 
    \end{tabular}  }
    \end{center}
      \end{subtable} 
      \\
   ~
    \\
      \begin{subtable}[t]{0.99\textwidth} 
    \begin{center}  \caption{Independent tests with large number of hypotheses}   \label{tab:sim-3} 
             \centerline{   \begin{tabular}{l|rrr|rrr}
    &\multicolumn{3}{c|}{$K=20,000$, $K_0=10,000$} &\multicolumn{3}{c}{$K=20,000$, $K_0=19,000$}\\
       \rule{0pt}{3ex}     & \multicolumn{1}{c}{$\alpha= 10\%$} &  \multicolumn{1}{c}{$\alpha=5\%$} & 
        \multicolumn{1}{c|}{$\alpha=2\%$} &  \multicolumn{1}{c}{$\alpha=10\%$} &  \multicolumn{1}{c}{$\alpha=5\%$} &  \multicolumn{1}{c}{$\alpha=2\%$} \\
      \hline 
BH       & 9567 (5.00) & 8564 (2.49) &7164 (1.00)  &   681.3    (9.58)  & 520.2 (4.79) &357.7 (1.93)\\
e-BH PRDS & 9092 (3.60) &8330 (2.13) &7124 (0.98) &  681.3 (9.58) &509.3 (4.54) &312.1 (1.40)\\\hline
BY     &   5956 (0.48) &4818 (0.24) &3417 (0.10) &  254.1 (0.89) &177.6 (0.46) &103.1 (0.19)\\
e-BH AD  & 6811 (0.80) &5809 (0.44) &4384 (0.18) &271.0 (1.02) &159.5 (0.39) & 51.4 (0.07)\\\hline
base  e-BH      &6426 (0.64) &5234 (0.31) &3509 (0.10) &224.8 (0.69) &109.2 (0.21)  &16.4 (0.01) \\ 
    \end{tabular}  }
  \end{center}
      \end{subtable}  
      \\~~ \\ 
      \begin{subtable}[t]{0.98\textwidth} 
    \begin{center}  \caption{Negatively correlated tests, $K=1000$, $K_0=800$.}     \label{tab:sim-5}  
        \centerline{   \begin{tabular}{l|rrr|rrr}
    &\multicolumn{3}{c|}{$ \rho_{ij}=-1/(K-1)$} &\multicolumn{3}{c}{$\rho_{ij}=-0.5\id_{\{|i-j|=1\}}$}\\
       \rule{0pt}{3ex}     & \multicolumn{1}{c}{$\alpha= 10\%$} &  \multicolumn{1}{c}{$\alpha=5\%$} & 
        \multicolumn{1}{c|}{$\alpha=2\%$} &  \multicolumn{1}{c}{$\alpha=10\%$} &  \multicolumn{1}{c}{$\alpha=5\%$} &  \multicolumn{1}{c}{$\alpha=2\%$} \\
      \hline 
BH        & 177.7 (8.14) &149.0 (4.09) &115.2 (1.61) &177.2 (8.10) &148.8 (4.00) &115.3 (1.62)\\
e-BH PRDS &172.0 (7.13) &147.9 (3.98) &114.9 (1.59) &171.5 (7.13) &147.7 (3.89) &114.9 (1.61)\\\hline
BY      &  101.2 (1.08) & 78.8 (0.52) & 53.3 (0.20) &101.3 (1.11) & 78.8 (0.56) & 53.2 (0.22)\\
e-BH AD &  109.7 (1.38) & 85.5 (0.65) & 54.6 (0.22) &109.8 (1.40) & 85.6 (0.69) & 54.6 (0.24)\\\hline
base e-BH      & 97.8 (0.98) & 70.7 (0.40) & 37.2 (0.11) &97.6 (0.99) & 70.7 (0.41) & 36.7 (0.12)\\ 
    \end{tabular}  }  \end{center} \end{subtable}  
  \end{center} 
 \end{table}

\section{Omitted proofs from the paper} \label{app:pf}

 \begin{proof}[Proof of Proposition \ref{prop:r1-1}]
 For $t\in [0,\infty)$, 
 let $f(t)$ be the number of true null hypotheses with an e-value $e'_k$ larger than or equal to $t$.
  Define the quantity
 \[ 
 Q(t) =  t{R(t)}/K,
 \]
 and it is clear that 
 $$
 t_\alpha = \inf \{t\in [0,\infty):Q(t) \ge 1/\alpha\}.
 $$
 Clearly $t_\alpha \in [1/\alpha,K/\alpha]$ since $Q(t)\le t$ and $R(t)\ge 1$.
Since $Q$ only has downside jumps and $Q(0)=0$, we know $Q(t_\alpha)=1/\alpha$, and thus
\begin{align}\label{eq:e1prime}
t_\alpha R(t_\alpha)= \frac{K}{\alpha}.
\end{align} 

If $e'_{k}\ge t_\alpha$, then $H_k$ is rejected by the definition of e-BH.
 If $e'_{[k]} < t_\alpha$, then by definition of $Q$, we have 
 $$\frac{k e'_{[k]}}{K}  \le \frac{R(e'_{[k]}) e'_{[k]}}{K} =Q(e'_{[k]}) < 1/\alpha.$$
 Thus, each $H_k$ is rejected by $\cG(\alpha)$ if and only if $e'_k\ge t_\alpha$.
 \end{proof}

 \begin{proof}[Proof of Lemma \ref{lem:e-FDR}]
 Note that   $f(\mathbf X)\le X_k'$   on the event $\{X_k\ge f(\mathbf X)\}$.
 \begin{enumerate}[(i)]
\item  If $X_k$ is independent of $\mathbf X^{-k}$, then 
\begin{align*}
\E\left[f(\mathbf X) \id_{\{X_k\ge f(\mathbf X)\}} \mid \mathbf X^{-k}   \right]
&\le
\E\left[X_k' \id_{\{X_k\ge f(\mathbf X)\}}  \mid \mathbf X^{-k}   \right]    \le \E\left[X_k'\right] \le \E[X_k].
\end{align*}
\item  
Let 
 $g :x\mapsto \p( X_k'\ge x)$  and $P_k:=g(X_k')$. Clearly $P_k$ is a null p-value.  
For each $y\ge 0$, 
since $\mathbf X$ is PRDS on $X_k$ and $g\circ f$ is an increasing function,  the function $
x \mapsto \p( g\circ f(\mathbf X) \le y \mid X_k \ge x ) 
$
is decreasing on $[0,\infty)$.
Noting that $P_k$ is a decreasing function of $X_k$, 
the function $
t\mapsto \p( g\circ f(\mathbf X) \le y \mid P_k\le t) 
$
is increasing on $[0,1]$. 
Using Lemma 1 of \cite{RBWJ19}, we have  
 $$   \E\left[\frac{   \id_{\{ P_k\le g \circ   f(\mathbf X )\}} }{ g \circ  f(\mathbf X)    }  \right]   \le 1.$$ 
It follows that  \begin{align*}
\E[f(\mathbf X) \id_{\{E_k \ge f(\mathbf X)\}}] &
=  \E[  f(\mathbf X)  \id_{\{ E_k' \ge f(\mathbf X)\}}] 
\\&
\le \E[ f(\mathbf X)   \id_{\{ P_k\le g \circ  f(\mathbf X)\}}]
\\&
=  \E\left[   f(\mathbf X)    \p( E_k'\ge f(\mathbf X))    \frac{   \id_{\{ P_k\le g \circ  f(\mathbf X) \}} }{ g \circ    f(\mathbf X) }   \right] 
\\&
\le  \sup_{x \ge  0} x \p(X'_k\ge x)   \E\left[\frac{   \id_{\{ P_k\le g \circ  f(\mathbf X)   \}} }{ g \circ  f(\mathbf X)    }  \right]   \\ &\le   \sup_{x \ge  0} x \p(X'_k\ge x).
\end{align*} 
\item It follows directly that
$$
\E\left[f(\mathbf X) \id_{\{X_k\ge f(\mathbf X)\}}   \right]
\le
\E\left[X_k' \id_{\{X_k\ge f(\mathbf X)\}}   \right]  \le \E[ X_k' ] \le \E[X_k].
$$
\end{enumerate}
If $X_k$ is a null e-value, then $\E[X_k]\le1$ by definition.
 \end{proof}

\begin{proof}[Proof of Theorem \ref{th:e2}]

By Proposition \ref{prop:r1-1} and \eqref{eq:e1prime},  we obtain that the FDR of e-BH equals
\begin{align}\label{eq:e1}
\E\left[\frac{F_{\cG(\alpha)}}{R_{\cG(\alpha)} }\right] =
\E\left[\frac{ f(t_\alpha) }{R(t_\alpha)}\right]  
=\frac{\alpha}{ K}
\E\left[ t_\alpha f(t_\alpha) \right].
\end{align} 
We will apply Lemma~\ref{lem:e-FDR}   to  $\mathbf X=\mathbf E'$ and the function $f$ given by $f(\mathbf X)=t_\alpha$, where $t_\alpha$ is treated as a function of $\mathbf E'$.
Note that $I_f=\alpha^{-1} K/\mathcal K$ by 
  \eqref{eq:e1prime}.
In this  case, $   X_k' =  T(\alpha X_k)/\alpha $. 
 By
 Lemma~\ref{lem:e-FDR} (iii) and \eqref{eq:e1}, we have
 \begin{align*}
\E\left[\frac{F_{\cG(\alpha)}}{R_{\cG(\alpha)} }\right]   =\frac{\alpha}{ K}
\E\left[ \sum_{k\in \mathcal N}  t_\alpha \id_{\{E'_k\ge t_\alpha\}} \right]  \le \frac{1}{ K}\E\left[ \sum_{k\in \mathcal N}   X'_k \right]  
= \frac{ K_0}{ K}y_\alpha,
\end{align*}
thus showing the first inequality. 
If $E_1',\dots,E_K'$ are the raw e-values or  the boosted e-values obtained by \eqref{eq:enhance-2} or \eqref{eq:enhance-2p}, then $y_\alpha\le \alpha$ by construction.
 \end{proof}
 
 \begin{proof}[Proof of Theoroem \ref{th:e1}]
Write $\mathbf E'= (E'_1,\dots,E'_K)$.
For each $k\in \mathcal N$, denote by 
$$
z_{\alpha,k} :=\max_{x\in K/\mathcal K}  x  {\p(\alpha E'_k \ge x)}
=\max_{x\ge 1}  x \p(T(\alpha E'_k)\ge x) \le \E[T(\alpha E'_k)].
$$ 
Note that  $t_\alpha = f(\mathbf E')$ for a decreasing function $f$ with range $I_f=\alpha^{-1}K/\mathcal K$. 
Using the notation in Lemma~\ref{lem:e-FDR} with $\mathbf X=\mathbf E'$, we have 
$$X_k' =   \sup \{x\in I_f \cup\{0\}: x\le   E'_k\} =\frac{T(\alpha E'_k)}{\alpha} .$$ 
  Lemma~\ref{lem:e-FDR} (ii) gives 
$$\E[t_\alpha \id_{\{E'_k \ge t_\alpha\}}]  \le \sup_{x \ge 0} x \p(X'_k\ge x)
= \max_{x\ge 1}  \frac x \alpha  \p(T(\alpha E'_k)\ge x)  =  \frac{z_{\alpha,k}}{\alpha}.
 $$ 
Hence, by \eqref{eq:e1}, $$\E\left[\frac{F_{\cG(\alpha)}}{R_{\cG(\alpha)} }\right]= \frac{\alpha}{K} \E\left[ t_\alpha f(t_\alpha) \right]  = \frac{\alpha}{K}  \sum_{k\in \mathcal N}  \E[t_\alpha \id_{\{E'_k \ge t_\alpha\}}] \le    \frac{1}{K} \sum_{k\in \mathcal N}z_{\alpha,k} .$$
Note that  $z_\alpha \le y_\alpha $ by Markov's inequality. Hence, if $E'_1,\dots,E'_K$ are raw e-values, then $z_\alpha \le y_\alpha \le \alpha$.
If $E'_1,\dots,E'_K$ are boosted e-values, then $z_\alpha \le \alpha$ by construction. 
\end{proof}

\begin{proof}[Proof of Proposition \ref{prop:bk}]
Since $b_k =(\alpha q_{1-\alpha} (E_k))^{-1}$, \eqref{eq:e-cond} implies 
 $$
t\mapsto t \p( b_k  E_k\ge t) \mbox{~is decreasing on $[1/ \alpha ,\infty)$}.
 $$
Therefore, $z_{\alpha,k} (b_k) $ in \eqref{eq:zalphak} satisfies 
$z_{\alpha,k} (b_k) = \p(\alpha b_k E_k \ge 1)$.
Direct calculation gives 
 $$z_{\alpha,k} (b_k) = \p(\alpha  E_k \ge 1/b_k) = \p(\alpha  E_k \ge \alpha q_{1-\alpha} (E_k) )=\alpha.$$
 Since $E_k$ is continuously distributed, $b_k$ cannot be enlarged without violating \eqref{eq:enhance-1}.
\end{proof}

\begin{proof}[Proof of Theorem \ref{th:optimal}] 
Define 
$$\tau_\phi=\inf\{\tau \ge 0 : \phi(\tau) R(\tau) \ge K \}.$$
Similarly to Proposition \ref{prop:r1-1}, 
 $\cG(\phi)$ rejects $H_k$ if and only if $e_k\ge \tau_\phi$.  By Theorem \ref{th:e2}, the FDR of $\cG(\phi)$ satisfies 
\begin{equation}
\label{eq:general-e}\E\left[\frac{F_{\cG (\phi)}}{R_{\cG(\phi)} }\right]  = \frac{1}{ K}  \sum_{k\in \mathcal N} \E[\phi(\tau_\phi) \id_{\{E_k \ge\tau_\phi \}}].
 \end{equation}

First, from \eqref{eq:e-k2},
only the values of $\phi^{-1}(K/k)$  for $ k\in \mathcal K $ affects the testing procedure $\cG(\phi)$. 
Note that   $\cG(\alpha)$ uses the linear transform   $t\mapsto \alpha t$.
Therefore, if $\phi^{-1}(  x)\ge   x/\alpha$ for all $x\in K/\mathcal K$, then  $\cG(\phi)$ is dominated by the e-BH procedure $\cG(\alpha)$, i.e., $\cG(\phi)\subseteq \cG(\alpha)$.

Next, suppose for the purpose of contradiction that $\phi^{-1}(K/k_0) <K/(k_0 \alpha)  $ for some $k_0\in \mathcal K$. Write 
$$m =\phi^{-1} (K/k_0) 
\mbox{~~~and~~~} m' = \max \{m, K/k_0\}.$$ 
Note that $m\le m' < K/(k_0\alpha)$. 
Consider the following setting of e-values:
Let    $\mathcal N=\mathcal K$, i.e., all e-values are null.  
For $k\in \mathcal K$, set
$E_k= m' \id_{A_k}$ for some set $A_k$ with $\p(A_k)=1/m'$.
We design the sets $A_k$, $k\in \mathcal K$, in the following way:
for each $\omega\in \Omega$, $$|\{k\in \mathcal K: \omega\in A_k\}|=k_0,$$ i.e., exactly $k_0$ events out of $A_1,\dots,A_K$ occur together (the existence of such an arrangement is easy to verify).
Let $A=\bigcup_{k\in \mathcal K} A_k$, and we have $\p(A)=K/(k_0 m') \in (\alpha,1]$. 

If event $A$ happens, then exactly $k_0$ of $E_1,\dots,E_K$ take the value $m'$, and the rest take the value $0$.
 Hence, 
for $\tau>0$, we have $R(\tau)=\max( k_0\id_{\{\tau\le m'\}} \id_A,1)$.
If event $A$ happens, then  $\tau_\phi \le m'$ since $  \phi(m')R(m') 
\ge \phi(m) k_0 =  K$, and   $\tau_\phi\ge m$ since $\phi (\tau_\phi) \ge K/k_0$. 
Therefore, $\tau_\phi \in [m,m']$ if $A$ happens. 
 It follows that, for each $k$, 
 $$\E[\phi(\tau_\phi) \id_{\{E_k \ge\tau_\phi \}}]
 \ge 
 \E[\phi(m) \id_{\{E_k \ge m' \}} \id_{A_k} ]
 = \E[\phi(m) \id_{A_k}]=\frac{K}{k_0m'} > \alpha.
 $$
 Using the FDR guarantee $\alpha$ of $ \cG (\phi) $  and \eqref{eq:general-e}, this leads to 
 $$ \alpha  \ge \E\left[\frac{F_{\cG (\phi)}}{R_{\cG(\phi)} }\right]  =  \frac{1}{ K}  \sum_{k\in \mathcal K} \E[\phi(\tau_\phi) \id_{\{E_k \ge\tau_\phi \}}] >  \alpha,
 $$
 a contradiction.
 Hence, $\cG(\phi)$ is dominated by $\cG(\alpha)$.   
\end{proof}

\begin{proof}[Proof of Proposition~\ref{prop:psi}]

For $k\in \mathcal N$, let $\alpha = \E[T(\psi(P))]$  where $P$ is uniform on $[0,1]$.
If $\psi^{-1}(1)=0$,  then $\p(\psi(P_k)\ge 1)=0$ for $k\in \mathcal N$,
and thus $\cD(\psi)$ will never reject any hypotheses with a non-zero p-value,  leading to a zero FDR for both PRDS and arbitrarily dependent p-values.  
Next, assume $\psi^{-1}(1)>0$, which implies $\alpha >0$.

Define
\begin{equation}\label{eq:iide}E_k= \frac{T(\psi(P_k))}{\alpha}, ~~~~k\in \mathcal K,\end{equation}
which is a null e-value for $k\in \mathcal N$. 
Note that 
$$
 \psi(p_{(k)}) \ge \frac{K}{k } ~~~~ \Longleftrightarrow~~~~ e_{[k]} \ge \frac{K}{k \alpha}. 
$$
Hence, $\cG(\alpha)$ applied to $(e_1,\dots,e_K)$ is equivalent to $\cD(\psi)$ applied to $(p_1,\dots,p_K)$.
Note that, using the notation in  Sections~\ref{sec:eBH-arbitrary} and~\ref{sec:eBH-PRDS},  
\begin{align*}y_\alpha  
 &=\frac 1 {   K_0} \sum_{k\in \mathcal N}\E\left[   T( \psi(P_k))\right] 
  \le \sum_{j=1}^K \frac{K}{j}
\left(\psi^{-1}\left(\frac{K}{j}\right) - \psi^{-1}\left(\frac{K}{j-1}\right) \right) 
 = y_\psi, \end{align*}
 and
 $$z_\alpha = \max_{x\in K/\mathcal K}  x \p( \psi(P_k) \ge x) \le 
 \max_{x\in K/\mathcal K}  x   \psi^{-1}(x) 
 = z_\psi. $$ 
By Theorems~\ref{th:e2} and~\ref{th:e1},  the testing procedure $\cD(\psi)$ satisfies 
$$
\E\left[\frac{F_{\cD(\psi)}}{R_{\cD(\psi)} }\right] 
=\E\left[\frac{F_{\cG(\alpha)}}{R_{\cG(\alpha)} }\right] \le
\frac{ K_0}{K}y_\alpha \le  \frac{    K_0}{K} y_{\psi},
$$ 
and $y_\psi$ can be replaced by $z_\psi$ if the p-values are PRDS. 
\end{proof}

\begin{proof}[Proof of Proposition~\ref{prop:psi3}]
Consider the setup where all 
all p-values are null, and they are identical following a uniform distribution on $[0,1]$. Note that identical p-values satisfy PRDS. 
In this case, ${F_{\cD(\psi)}}/{R_{\cD(\psi)} } =1$ as soon as there is any discovery. 
 If $\psi^{-1}(1) >\alpha$, then 
$$\p(\psi (p_{(K)}) \ge 1) = \p( p_{(K)} \le \psi^{-1}(1) ) = \min\{\psi^{-1}(1),1\} > \alpha.$$
Hence, the probability of having a false discovery is more than $\alpha$, violating the assumption that $\E [ {F_{\cD(\psi)}}/{R_{\cD(\psi)} } ] \le  \alpha.$ Therefore, $\psi^{-1}(1) \le \alpha$.

If $\psi$  satisfies \eqref{eq:e-cond2}
 and $\psi^{-1}(1)\le \alpha$, 
 then we have $t\psi^{-1}(t)\le \alpha$ for all $t\ge 1$.
 Hence, $\psi(p)\le \alpha/p$ for $p\in (0,\alpha]$.
As a consequence, $\cD(\psi)\subseteq \cD(\alpha)$.  
\end{proof}

\section{Using e-BH with multiple decreasing transforms} \label{app:multi}

We can generalize the p-testing procedure $\cD(\psi)$ in Section~\ref{sec:gen-p} to the case of multiple decreasing transforms, which we briefly describe here.

Let $\psi_1,\dots,\psi_K$ be decreasing transforms,
and write $r_k:=\psi_k(p_k)$, $k\in \mathcal K$.
Design a p-testing procedure $\cD(\psi_1,\dots,\psi_K)$ by rejecting the $k ^*$ hypotheses with the largest $r_k$, where
$$
k^*=\max\left\{k\in \mathcal K:  r_{[k]} \ge \frac{K}{k }\right\},$$
with the convention $\max(\varnothing)=0.$ 
Here, we are not rejecting hypotheses with the smallest p-values (a hypothesis with a smaller p-value may not be rejected before one with a larger p-value), but rather those with  the largest values of $r_k$; thus the procedure is no longer a step-up procedure in Section \ref{sec:gen-p}.
Recall that for a decreasing transform $\psi$,  $y_\psi$ is defined by \eqref{eq:ypsi},
and  $z_\psi$ is defined by \eqref{eq:zpsi}.
The following   result is a stronger version  of Proposition ~\ref{prop:psi}.
 
\begin{proposition}
\label{prop:psi-multi}
  For arbitrary p-values and   decreasing transforms $\psi_1,\dots,\psi_K$,  
   the testing procedure $\cD(\psi_1,\dots,\psi_K)$ satisfies 
$$
\E\left[\frac{F_{\cD(\psi_1,\dots,\psi_K)}}{R_{\cD(\psi_1,\dots,\psi_K)} }\right] \le \frac{1}{K}  \sum_{k\in \mathcal N} y_{\psi_k}.
$$ 
If the p-values are PRDS, then
$$
\E\left[\frac{F_{\cD(\psi_1,\dots,\psi_K)}}{R_{\cD(\psi_1,\dots,\psi_K)} }\right] \le \frac{   1}{K}\sum_{k\in \mathcal N} z_{\psi_k}.
$$ 
\end{proposition}  
\begin{proof}
The proof is similar to that of Proposition~\ref{prop:psi}. 
  Define
$$E_k= \frac{T(\psi_k(P_k))}{\beta }, ~~~~k\in \mathcal K,$$
where  $\beta = \max_{k\in \mathcal N} \E[T(\psi_k(P_k))]$.
Clearly, $E_k$ 
  is an e-value for $k\in \mathcal N$. 
Note that 
$$
r_{[k]} \ge \frac{K}{k } ~~~~ \Longleftrightarrow~~~~ e_{[k]} \ge \frac{K}{k \beta}. 
$$
Therefore, $\cG(\beta)$ applied to $(e_1,\dots,e_K)$ is equal to $\cD(\psi_1,\dots,\psi_K)$ applied to $(p_1,\dots,p_K)$. 
Using Theorems~\ref{th:e2} and~\ref{th:e1},  $\cD(\psi_1,\dots,\psi_K)$ satisfies 
\begin{align*}
\E\left[\frac{F_{\cD(\psi_1,\dots,\psi_K)}}{R_{  \cD(\psi_1,\dots,\psi_K)} }\right] 
& =\E\left[\frac{F_{\cG(\beta)}}{R_{\cG(\beta)} }\right]    \le
\frac{ K_0}{K}y_\beta    = 
 \frac{   1}{K}\sum_{k\in \mathcal N} y_{\psi_k},
  \end{align*} 
and $y_{\psi_k}$ can be replaced by $z_{\psi_k}$ if the p-values are PRDS. 
\end{proof}

\end{document}